\documentclass[a4paper,12pt]{article}

\usepackage[left=2cm,right=2cm, top=2cm,bottom=3cm,bindingoffset=0cm]{geometry}

\usepackage{verbatim}
\usepackage{amsmath}
\usepackage{amsthm}
\usepackage{amssymb}
\usepackage{delarray}
\usepackage{cite}
\usepackage{hyperref}
\usepackage{mathrsfs}
\usepackage{tikz}
\usetikzlibrary{patterns}
\usepackage{caption}
\DeclareCaptionLabelSeparator{dot}{. }
\newcommand{\al}{\alpha}

\newcommand{\ga}{\gamma}
\newcommand{\de}{\delta}
\newcommand{\la}{\lambda}
\newcommand{\om}{\omega}

\newcommand{\eps}{\varepsilon}

\newcommand{\iy}{\infty}

\theoremstyle{plain}

\numberwithin{equation}{section}

\newtheorem{thm}{Theorem}[section]
\newtheorem{lem}[thm]{Lemma}
\newtheorem{prop}[thm]{Proposition}
\newtheorem{cor}[thm]{Corollary}

\theoremstyle{definition}

\newtheorem{ip}[thm]{Inverse Problem}
\newtheorem{df}[thm]{Definition}

\theoremstyle{remark}

\DeclareMathOperator*{\Res}{Res}
\DeclareMathOperator{\rank}{rank}
\DeclareMathOperator{\diag}{diag}
\DeclareMathOperator{\Ran}{Ran}

 \allowdisplaybreaks

\begin{document}

\begin{center}
{\large\bf Spectral data characterization for the Sturm-Liouville operator on the star-shaped graph}
\\[0.2cm]
{\bf Natalia P. Bondarenko} \\[0.2cm]
\end{center}

\vspace{0.5cm}

{\bf Abstract.} The inverse spectral problems are studied for the Sturm-Liouville operator on the star-shaped graph and for the matrix Sturm-Liouville operator with the boundary condition in the general self-adjoint form. We obtain necessary and sufficient conditions of solvability for these two inverse problems, and also prove their local solvability and stability.
  
\medskip


\medskip


\vspace{1cm}

\section{Introduction}

The paper aims to give spectral data characterization for the Sturm-Liouville operator on a geometrical graph. Differential operators on graphs, also called quantum graphs, are used for modeling wave propagation in structures, consisting of thin tubes, strings, beams, etc. Such models appear in organic chemistry, mechanics, nanotechnology, theory of waveguides and other applications (see \cite{Nic85, LLS94, Kuch02, GS06, BCFK06, BK13, PPP04} and references therein).

Inverse spectral problems, that consist in recovering differential operators on graphs from their spectral characteristics, have been studied by many scholars (see \cite{Piv00, KS00, GS01, Har02, Bel04, KN05, BW05, Yur05, BV06, Piv07, CP07, AK08, FIY08, TMM06, Kur10, Yang10, ALM10, AKN10, EK12, BF13, Ign15, Yur16, BS17, MT17, XY18}). The results of those studies generalize the classical results of inverse problem theory for ordinary differential operators on intervals (see the monographs \cite{Mar77, Lev84, PT87, FY01}). The majority of the papers on inverse problems for quantum graphs concern the second-order (also called Sturm-Liouville or Schr\"odinger) differential operators. On the one hand, such operators are easier for investigation, on the other hand, they are natural for applications. 

For quantum graphs, there are three types of inverse problems, that consist in reconstruction of the following characteristics: 

\begin{enumerate}
\item Coefficients of differential expressions (e.g., Sturm-Lioville potentials) on the graph edges (see \cite{Piv00, Har02, Bel04, BW05, Yur05, BV06, Piv07, AK08, FIY08, Yang10, Kur10, BF13, Ign15, Yur16, BS17, MT17, XY18}).
\item Graph structure (see \cite{GS01, KN05, Kur10, ALM10}). 
\item Boundary conditions (see \cite{KS00, AKN10, EK12}). 
\end{enumerate}

In this paper, we focus on a problem of the first type. For recovering coefficients of differential expressions on graphs, two constructive methods made the most significant impact. The first of them is the BC-method, developed by Belishev and his successors (see \cite{Bel04, BV06, AK08, ALM10}). That method allowed them to solve inverse problems on arbitrary trees (graphs without cycles) and recover not only operator coefficients, but also a graph structure. The second approach is based on the method of spectral mappings (see \cite{Yur05, FIY08, BF13, Ign15, Yur16, BS17}). Relying on that method, Yurko and other mathematicians have solved inverse spectral problems for differential operators on arbitrary compact graphs (see \cite{Yur16}) and inverse spectral-scattering problems on noncompact graphs (see \cite{BF13, Ign15, Yur16}). There were also attempts to apply the methods of Marchenko (see \cite{Mar77, AM63}) to inverse scattering problems for special types of graphs with infinite rays (see \cite{Har02, TMM06, MT17, XY18}). Nevertheless, although there is a significant number of studies on inverse problems for differential operators on graphs, they concern only uniqueness theorems and constructive algorithms for solution. The question of spectral data characterization remained open even for the following operator on the simplest star-shaped graph.

In this paper, we consider the geometrical graph $G$ with the vertices $\{ v_j \}_{j = 0}^m$ and the edges $\{ e_j \}_{j = 1}^m$. Every edge $e_j$ connects the vertices $v_0$ and $v_j$, $j = \overline{1, m}$, i.e. $v_0$ is the internal vertex, and $\{ v_j \}_{j = 1}^m$ are the boundary vertices. For $j = \overline{1, m}$, we associate with the edge $e_j$ the interval $[0, \pi]$ and a parameter $x_j \in [0, \pi]$, so that $x_j = 0$ corresponds to the boundary vertex $v_j$ and $x_j = \pi$ corresponds to the internal vertex $v_0$. 

Consider the system of the Sturm-Liouville equations on the star-shaped graph $G$:
\begin{equation} \label{eqv-g}
    \ell_j y_j := -y_j''(x_j) + q_j(x_j) y_j(x_j) = \la y_j(x_j), \quad x_j \in (0, \pi), \quad j = \overline{1, m},
\end{equation}
with the Dirichlet conditions at the boundary vertices
\begin{equation} \label{bc-g}
    y_j(0) = 0, \quad j = \overline{1, m},
\end{equation}
and the following matching conditions at the internal vertex
\begin{equation} \label{mc-g}
    y_1(\pi) = y_j(\pi), \quad j = \overline{1, m}, \quad
    \sum_{j = 1}^m (y_j'(\pi) - h y_j(\pi)) = 0.
\end{equation}

Here $q_j$, $j = \overline{1, m}$, are real-valued functions from $L_2(0, \pi)$, called {\it the potentials}, and $h \in \mathbb R$.

Introduce the spaces
\begin{gather*}
L_2(G) := \{ y = [y_j]_{j = 1}^m \colon y_j \in L_2(0, \pi), \, j = \overline{1, m} \}, \\
W_2^2(G) := \{ y = [y_j]_{j = 1}^m \colon y_j, y_j' \in AC[0, \pi], \, y_j'' \in L_2(0, \pi), \, j = \overline{1, m} \}.
\end{gather*}

The boundary value problem \eqref{eqv-g}-\eqref{mc-g} defines the self-adjoint operator $\mathcal L$ in $L_2(G)$, acting by the rule $\mathcal L y = [\ell_j y_j]_{j = \overline{1, m}}$ and having the domain
$$
\mathcal D(\mathcal L) := \{ y \in W_2^2(G) \colon y \: \text{satisfies (1.2), (1.3)} \}.
$$

It is well-known that the operator $\mathcal L$ has a purely discrete spectrum, consisting of real eigenvalues.

\begin{df} \label{def:eig}
Let $\{ \la_{nk} \}_{n \ge 1, \, k = \overline{1,m}}$ be the eigenvalues of $\mathcal L$, numbered in the nondecreasing order: $\la_{n_1, k_1} \le \la_{n_2, k_2}$, if $(n_1, k_1) < (n_2, k_2)$, i.e. $n_1 < n_2$ or $n_1 = n_2$, $k_1 = k_2$. Multiple eigenvalues occur in the sequence $\{ \la_{nk} \}_{n \ge 1, \, k = \overline{1, m}}$ several times, according to their multiplicities. It is convenient to number the eigenvalues by two indices $n$ and $k$ because of the asymptotic formulas \eqref{asymptla}.
\end{df}

\begin{df} \label{def:Weyl}
For $k = \overline{1, m}$, we introduce the vector function $\Phi_k(x, \la) = [\phi_{kj}(x, \la)]_{j = 1}^m$, satisfying equations~\eqref{eqv-g} for $x_j = x$, $j = \overline{1, m}$, the matching conditions~\eqref{mc-g} and the following boundary conditions:
$$
    \phi_{kk}(0, \la) = 1, \quad \phi_{kj}(0, \la) = 0, \quad k, j = \overline{1, m}, \: k \ne j.
$$
Let $\Phi(x, \la) = [\phi_{kj}(x, \la)]_{k, j = 1}^m$ be the matrix with the columns $\Phi_k(x, \la)$. The matrix function $M(\la) := \Phi'(0, \la)$ is called {\it the Weyl matrix} of the boundary value problem \eqref{eqv-g}-\eqref{mc-g}.
\end{df}

Weyl matrices and Weyl functions are natural spectral characteristics for recovering differential operators of various types (see, e.g., \cite{BW05, Yur16, FIY08, Ign15, Mar77, FY01, Xu19}).

The matrix functions $\Phi(x, \la)$ and $M(\la)$ are meromorphic in the $\la$-plane. All their poles are simple and coincide with the the eigenvalues $\{ \la_{nk} \}_{n \ge 1, \, k = \overline{1, m}}$ (see \cite{Xu19, Bond19-sp}).
Thus, we define the weight matrices
\begin{equation} \label{defal}
    \al_{nk} := -\Res_{\la = \la_{nk}} M(\la), \quad n \ge 1, \quad k = \overline{1, m}.
\end{equation}

The collection $\mathcal S := \{ \la_{nk}, \al_{nk} \}_{n \ge 1, \, k = \overline{1, m}}$ is called {\it the spectral data} of $\mathcal L$. This paper is devoted to the following inverse spectral problem.

\begin{ip} \label{ip:graph}
Given the spectral data  $\mathcal S$, construct the potentials $\{ q_j \}_{j = 1}^m$ and the coefficient $h$.
\end{ip}

The uniqueness of Inverse Problem~\ref{ip:graph} solution follows, in particular, from the results of \cite{BW05, Yur05, Xu19, Bond19-alg}. In the papers \cite{Yur05, FIY08, Bond19-alg}, a constructive solution of this inverse problem has been developed, based on the method of spectral mappings \cite{FY01}.

In this paper, we obtain necessary and sufficient conditions of solvability for Inverse Problem~\ref{ip:graph}. In other words, we provide spectral data characterization for the Sturm-Liouville operator $\mathcal L$ on the star-shaped graph. Moreover, local solvability and stability of Inverse Problem~\ref{ip:graph} are proved.

The question of necessary and sufficient is the most important issue of inverse problem theory and usually is the most complicated one. For differential operators on graphs, this question has not been solved before. Complicated structural properties and the behavior of the spectrum cause significant difficulties in spectral data characterization for quantum graphs.
Some results in this directions were obtained by Pivovarchik \cite{Piv00,Piv07}. However, for reconstruction of the operator, Pivovarchik used spectra corresponding to separate edges of the graph but not to the whole graph. Local solvability means that the solution of inverse problem still exists under a sufficiently small perturbation of the spectral data. Local solvability is closely related with stability, which is essential for justification of numerical methods for solving inverse problems. 

Our approach is based on representation of the boundary value problem \eqref{eqv-g}-\eqref{mc-g} in the equivalent matrix form:
\begin{gather} \label{eqv}
    -Y''(x) + Q(x) Y(x) = \la Y(x), \quad x \in (0, \pi), \\ \label{bc}
    Y(0) = 0, \quad V(Y) := T (Y'(0) - H Y(0)) - T^{\perp} Y(0) = 0,
\end{gather}
where $Y(x) = [y_j(x)]_{j = 1}^m$ is a vector function, $Q(x) = \diag\{ q_j(x) \}_{j = 1}^m$ is the diagonal matrix, and
\begin{equation} \label{defT}
T = [T_{jk}]_{j, k = 1}^m, \quad T_{jk} = \tfrac{1}{m}, \: j, k = \overline{1, m}, \quad
T^{\perp} = I - T, \quad H = h T.
\end{equation}
The symbol $I$ denotes the $m \times m$ unit matrix. We denote the problem~\eqref{eqv}-\eqref{bc} by $L = L(Q(x), T, H)$.

In addition, we study the problem $L(Q(x), T, H)$ in {\it the general form}, where 
\begin{itemize}
    \item $Q(x)$ is an arbitrary Hermitian matrix function with the elements from $L_2(0, \pi)$;
    \item $T$ is an arbitrary orthogonal projector, $1 \le \rank(T) \le m-1$, $T^{\perp} = I - T$;
    \item $H$ is a Hermitian matrix, such that $H = T H T$. 
\end{itemize}
The case, when $Q(x)$ is diagonal and \eqref{defT} is fulfilled, is called {\it the graph case}.

Note that the condition $V(Y) = 0$ turns into the Dirichlet condition $Y(\pi) = 0$ in the case $T = 0$ and into the Robin condition $Y'(\pi) - H Y(\pi) = 0$ in the case $T = I$. In the latter two degenerated cases, our main results remain valid, but the proofs require technical modifications. Therefore we suppose that $1 \le \rank(T) \le m - 1$. 

For the problem $L$ in the general case, we define {\it the Weyl solution} $\Phi(x, \la)$ as the matrix solution of equation~\eqref{eqv}, satisfying the conditions $\Phi(0, \la) = I$, $V(\Phi) = 0$, and {\it the Weyl matrix} as follows: $M(0, \la) = \Phi'(0, \la)$. Clearly, these definitions generalize Definition~\ref{def:Weyl} for the graph case. The weight matrices in the general case are defined by the formula~\eqref{defal}. Along with Inverse Problem~\ref{ip:graph}, we investigate the following general matrix inverse problem.

\begin{ip} \label{ip:matr}
Given the spectral data $\mathcal S$, find $Q(x)$, $T$ and $H$.
\end{ip}

The most complete investigation of inverse problems has been carried out for the matrix Sturm-Liouville equation~\eqref{eqv} with the Dirichlet boundary conditions $Y(0) = Y(\pi) = 0$ and the Robin boundary conditions $Y'(0) - H_1 Y(0) = 0$, $Y'(\pi) + H_2 Y(\pi) = 0$ instead of \eqref{bc}. Here $H_1$ and $H_2$ are $m \times m$ matrices. In \cite{CK09, MT10, Bond12, Bond19-tamkang}, spectral data characterization has been provided for those matrix Sturm-Liouville operators. Nevertheless, operators with general self-adjoint boundary conditions appeared to be more difficult for investigation. There are only uniqueness results for recovering the matrix Sturm-Liouville operator with the both boundary conditions in the form similar to $V(Y) = 0$ from spectral characteristics (see \cite{Xu19}). In the recent study \cite{Bond19-alg}, a constructive method for solving Inverse Problem~\ref{ip:matr} has been developed. We also mention that the inverse scattering problem for the matrix Sturm-Liouville operator on the half-line with the Dirichlet boundary condition at $x = 0$ was solved in \cite{AM63}. Harmer \cite{Har02} generalized the results of \cite{AM63} to the case of general self-adjoint boundary condition analogous to $V(Y) = 0$. In addition, Harmer \cite{Har02} studied the inverse scattering problem for the Sturm-Liouville operator on the star-shaped graphs with infinite rays. However, the operators considered in \cite{AM63, Har02} have a finite number of eigenvalues, so the difficulties related to spectral data asymptotics do not arise. Therefore inverse scattering problems for matrix Sturm-Liouville operators on infinite domains appear to be easier for investigation than inverse spectral problems on a finite interval.

In this paper, we obtain necessary and sufficient conditions of solvability for Inverse Problem~\ref{ip:matr} for the general matrix case and, in parallel, for Inverse Problem~\ref{ip:graph} for the Sturm-Liouville operator on the graph. Furthermore, local solvability and stability are proved for the both problems. Note that our necessary and sufficient conditions (Proposition~\ref{prop:nc}, Theorems~\ref{thm:sc} and~\ref{thm:scg}) generalize \cite[Theorem~1.6.2]{FY01} for the scalar Sturm-Liouville operator on a finite interval. Similarly, local solvability and stability Theorems~\ref{thm:loc} and~\ref{thm:locg} generalize \cite[Theorem~1.6.4]{FY01}. However, these generalizations are far from being trivial. The main difficulty in our research is caused by complicated behavior of the spectrum. The spectrum of the problem $L$ can contain an infinite number of groups of multiple and/or asymptotically multiple eigenvalues, that influences the structure of the weight matrices. In order to overcome this difficulty, we group the eigenvalues by asymptotics and investigate the sums of the weight matrices, corresponding to each group.

Our analysis relies on the basic ideas of the method of spectral mappings (see \cite{FY01}). A crucial step of this method is contour integration in the complex plane of the spectral parameter. As a result, a nonlinear inverse problem is reduced to a linear equation in a Banach space. Investigation of matrix Sturm-Liouville operators requires essential development of this method. We construct a special Banach space of infinite matrix sequences, by relying on our eigenvalue grouping, and then investigate solvability of the main equation in that Banach space.
    
The paper is organized as follows. Section~2 contains preliminaries. We provide asymptotic formulas for the eigenvalues $\{ \la_{nk} \}$ and for the weight matrices $\{ \al_{nk} \}$. Then Inverse Problem~\ref{ip:matr} is reduced to the so-called main equation in an appropriate Banach space. In Section~3, we formulate necessary and sufficient conditions of solvability for Inverse Problems~\ref{ip:graph} and~\ref{ip:matr}. The proofs are provided in the next three sections. In Section~4, auxiliary asymptotics and estimates are obtained. In Section~5, we investigate solvability of the main equation. In Section~6, using the solution of the main equation, we construct the operator and show that its spectral data coincide with the initially given numbers. In Section~7, local solvability and stability theorems are provided. 

\section{Preliminaries}

The goal of this section is to provide preliminary results from~\cite{Bond19-sp, Bond19-asympt, Bond19-alg}. In particular, Propositions~\ref{prop:asymptla} and~\ref{prop:asymptal} give asymptotic formulas for the eigenvalues and the weight matrices, respectively. Further the special Banach space $B$ of infinite matrix sequences is constructed, and Inverse Problem~\ref{ip:matr} is reduced to the main equation~\eqref{main} in $B$. In the construction of $B$, an important role is played by the grouping $\{ G_k \}_{k \ge 1}$ of the square roots of the eigenvalues. 

First of all, we introduce the {\bf notations.}

\begin{itemize}
    \item The prime denotes differentiation by $x$ in expressions similar to $Y'(x, \la)$.
    \item The symbol $\dagger$ denotes the conjugate transpose, i.e. for a matrix $A = [a_{jk}]_{j, k = 1}^m$ we have $A^{\dagger} = [\bar a_{kj}]_{j, k = 1}^m$.
    \item The spaces of complex-valued $m$-vectors and $m \times m$ matrices are denoted by $\mathbb C^{m}$ and $\mathbb C^{m \times m}$, respectively. In these spaces, we use
    the Euclidean vector norm and the induced matrix norm: $\| A \| = \sqrt{ \la_{max} (A^{\dagger} A)}$, where $\la_{max}$ is the maximal eigenvalue.
    \item For any interval $\mathbb I \subset \mathbb R$, we denote by $L_2(\mathbb I; \mathbb C^m)$ and $L_2(\mathbb I; \mathbb C^{m \times m})$ the spaces of $m$-vector functions and $m \times m$ matrix functions, respectively, having elements from $L_2(\mathbb I)$. For example, $Q \in L_2((0, \pi); \mathbb C^{m \times m})$. 
    \item The scalar product and the norm in the Hilbert space $L_2(\mathbb I; \mathbb C^{m \times m})$ are defined as follows:
$$
(Y, Z) = \int_{\mathbb I} Y^{\dagger}(x) Z(x) \, dx, \quad \| Y \| = \sqrt{(Y, Y)}, \quad
Y = [y_j(x)]_{j = 1}^m, \quad Z = [z_j(x)]_{j = 1}^m.
$$
    \item In $L_2(\mathbb I; \mathbb C^{m \times m})$, the following norm is used:
$$
\| A \|_{L_2} = \max_{1 \le j, k \le m} \left( \int_{\mathbb I} |a_{jk}(x)|^2 \right)^{1/2}, \quad A(x) = [a_{jk}(x)]_{j, k = 1}^m.
$$
    \item The matrix Wronskian is denoted by $\langle Y(x), Z(x)\rangle := Y(x) Z'(x) - Y'(x) Z(x)$, where $Y$ and $Z$ are $m \times m$ matrix functions.
    \item In estimates, we use the same symbol $C$ for various constants, independent of $x$, $\la$, $n$, etc.
    \item The notation $\{ K_n \}$ is used for various matrix sequences, such that $\{ \| K_n \| \} \in l_2$.
    \item $\la = \rho^2$, $\tau := \mbox{Im}\,\rho$.
\end{itemize}

Below we suppose that the problem $L = L(Q(x), T, H)$ is of the general from, unless the opposite is stated. Denote $p := \rank (T)$, then $\rank(T^{\perp}) = m - p$. In the general case, $1 \le p \le m - 1$. In the graph case, we have $p = 1$ according to \eqref{defT}. 

Denote
\begin{gather} \nonumber
\Omega := \frac{1}{2} \int_0^{\pi} Q(x) \, dx, \\ \label{defP12}
\mathcal P_1(z) := z^{p-m} \det(z I - T (\Omega - H) T), \quad
\mathcal P_2(z) := z^{-p} \det (z I - T^{\perp} H T^{\perp} ).
\end{gather}

One can easily show that $\mathcal P_1(z)$ and $\mathcal P_2(z)$ are polynomials of degrees $p$ and $(m - p)$, respectively, whose roots are real. Denote the roots of $\mathcal P_1(z)$ by $\{ z_k \}_{k = 1}^p$ and the roots of $\mathcal P_2(z)$ by $\{ z_k \}_{k = p + 1}^m$, counting with the multiplicities and in the nondecreasing order: $z_k \le z_{k + 1}$ for $k = \overline{1, m} \backslash \{ p \}$. 

In the graph case, we have $\Omega = \diag\{ \om_j \}_{j = 1}^m$, $\om_j := \frac{1}{2} \int_0^{\pi} q_j(x) \, dx$, $j = \overline{1, m}$,
\begin{equation} \label{defz1}
z_1 = \frac{1}{m} \sum_{j = 1}^m \om_j - h,
\end{equation}
and $\{ z_j \}_{j = 2}^m$ are the roots of the polynomial $\mathcal P_2(z)$, which takes the form
\begin{equation} \label{defP2}
    \mathcal P_2(z) = \frac{1}{m} \frac{d}{dz} \left( \prod_{j = 1}^m (z - \om_j) \right)
\end{equation}

Let $\{ \la_{nk} \}_{n \ge 1, \, k = \overline{1, m}}$ be the eigenvalues of $L$, numbered according to Definition~\ref{def:eig}. Put $\rho_{nk} := \sqrt{\la_{nk}}$, $n \ge 1$, $k = \overline{1, m}$. The following proposition gives the asymptotic formulas for the eigenvalues.

\begin{prop} \label{prop:asymptla}
The following relations hold
\begin{equation} \label{asymptla}
\arraycolsep=1.4pt\def\arraystretch{1.8}
\left.\begin{array}{ll}
\rho_{nk} = n - \dfrac{1}{2} + \dfrac{z_k}{\pi n}+ \dfrac{\varkappa_{nk}}{n},  & \quad k = \overline{1, p}, \\
\rho_{nk} = n + \dfrac{z_k}{\pi n} + \dfrac{\varkappa_{nk}}{n}, & \quad 
k = \overline{p + 1, m},
\end{array} \right\}
\end{equation}
where $n \ge 1$, $\{ \varkappa_{nk} \} \in l_2$.
\end{prop}

Proposition~\ref{prop:asymptla} has been proved in \cite{Bond19-asympt} for the general case and in~\cite{MP15} for the graph case.

In order to provide asymptotic formulas for the weight matrices $\{ \al_{nk} \}_{n \in \mathbb N, \, k = \overline{1, m}}$, we need some additional notations.

\begin{df} \label{def:al}
Consider in the sequence $\{ \la_{nk} \}_{n \ge 1, \, k = \overline{1, m}}$ any group of multiple eigenvalues $\la_{n_1, k_1} = \la_{n_2, k_2} = \dots = \la_{n_r, k_r}$, $(n_j, k_j) < (n_{j + 1}, k_{j + 1})$, $j = \overline{1, r-1}$, maximal by inclusion. Obviously, we have $\al_{n_1, k_1} = \al_{n_2, k_2} = \dots = \al_{n_r, k_r}$. Define $\al'_{n_1, k_1} := \al_{n_1, k_1}$ and $\al'_{n_j, k_j} := 0$ for $j = \overline{2, r}$. Defining the matrices $\al'_{nk}$ for every group of multiple eigenvalues in such a way, we get the sequence $\{ \al'_{nk} \}_{n \ge 1, \, k = \overline{1, m}}$.
\end{df}

Introduce the sums
\begin{gather*}
\al_n^I = \sum_{k = 1}^p \al'_{nk}, \quad \al_n^{II} = \sum_{k = p + 1}^m \al'_{nk}, \\
\al_n^{(s)} = \sum_{\substack{k = 1 \\ z_k = z_s}}^p \al'_{nk}, \: s = \overline{1, p}, \qquad
\al_n^{(s)} = \sum_{\substack{k = p+1 \\ z_k = z_s}}^{m} \al'_{nk}, \: s = \overline{p + 1, m}.
\end{gather*}

\begin{prop} \label{prop:asymptal}
The following relations hold
\begin{equation} \label{asymptal}
\arraycolsep=1.4pt\def\arraystretch{1.8}
\left.\begin{array}{c}
\al_n^I = \frac{2(n - 1/2)^2}{\pi} \Bigl( T + \frac{K_n}{n} \Bigr), \quad
\al_n^{II} = \frac{2 n^2}{\pi} \Bigl( T^{\perp} + \frac{K_n}{n}\Bigr), \\
\al_n^{(s)} = \frac{2 n^2}{\pi} (A^{(s)} + K_n), \quad s = \overline{1, m},
\end{array} \right\}
\end{equation}
where 
\begin{gather*}
A^{(s)} = U^{\dagger} T_s U, \quad s = \overline{1, m}, \\
T_s = [T_{s, jk}]_{j, k = 1}^m, \quad
T_{s, jk} = \begin{cases}
                1, \quad j = k, \quad z_j = z_s, \quad j, s \le p \quad \text{or} \quad j,s > p, \\
                0, \quad \text{otherwise},
             \end{cases}
\end{gather*}
and $U$ is a unitary matrix, such that
$$
U \Theta U^{\dagger} = \diag\{ z_k \}_{k = 1}^m, \quad
\Theta := T (\Omega - H) T + T^{\perp} \Omega T^{\perp}.
$$
The matrices $\{ A^{(s)} \}_{s = 1}^m$ do not depend on the choice of $U$.

In the graph case, the following relations hold
\begin{gather} \label{defAs}
A^{(1)} = T, \quad A^{(s)} = \frac{1}{m} \Res_{z = z_s} \frac{A(z)}{\mathcal P_2(z)}, \quad s = \overline{2, m},  \\ \label{defA}
A(z) = [a_{jk}]_{j,k = 1}^m, \quad a_{jj}(z) =\frac{d}{dz} \left( \prod_{\substack{s = 1 \\ s \ne j}}^m (z - \om_s) \right), 
\quad 
a_{jk} = -\prod_{\substack{s = 1 \\ s \ne j, k}}^m (z - \om_s), \quad j \ne k.
\end{gather}
\end{prop}

Proposition~\ref{prop:asymptal} has been proved in \cite{Bond19-asympt} for the general case and in~\cite{Bond19-sp} for the graph case.

Further we need the main equation of Inverse Problem~\ref{ip:matr}, derived in~\cite{Bond19-alg}. Consider a model problem $\tilde L = L(\tilde Q(x), \tilde T, \tilde H)$ of the same form as $L$, but with different coefficients. We agree that, if a certain symbol $\ga$ denotes an object related to $L$, the symbol $\tilde \ga$ with tilde denotes the analogous object related to $\tilde L$. Let $\tilde L$ be such that $\tilde T = T$ and $\tilde \Theta = \Theta$. In particular, one can put $\tilde Q(x) := \frac{2}{\pi} \Theta$, $\tilde T := T$, $\tilde H = 0$. A detailed algorithm for constructing the problem $\tilde L$ by using the spectral data is provided in~\cite{Bond19-alg}.

In the graph case, it is convenient to choose the model problem $\tilde L$ with the diagonal potential matrix. Suppose that we know the mean values $\{ \om_j \}_{j = 1}^m$. Then we put 
\begin{equation} \label{modelg}
\tilde Q(x) := \frac{2}{\pi} \Omega, \quad \tilde T := T, \quad \tilde H := \tilde h T, \quad \tilde h = \frac{1}{m} \sum\limits_{j = 1}^m \om_j - z_1.
\end{equation}

Denote by $S(x, \la)$ the $m \times m$ matrix solution of equation~\eqref{eqv}, satisfying the initial conditions $S(0, \la) = 0$, $S'(0, \la) = I$. Define
\begin{equation} \label{defD}
D(x, \la, \mu) := \frac{\langle S^{\dagger}(x, \bar \la), S(x, \mu) \rangle}{\la - \mu}.
\end{equation}

Without loss of generality, we can assume that $\la_{nk} \ge 0$ and $\tilde \la_{nk} \ge 0$ for all $n \in \mathbb N$, $k = \overline{1, m}$. One can easily achieve these conditions by a shift of the spectrum: $\la \mapsto \la + C$, $Q(x) \mapsto Q(x) + C I$, where $C$ is a constant. Introduce the notations:
\begin{gather*}
    \la_{nk0} = \la_{nk}, \quad \la_{nk1} = \tilde \la_{nk}, \quad \rho_{nk0} = \rho_{nk}, \quad \rho_{nk1} = \tilde \rho_{nk}, \\
    \al_{nk0} = \al_{nk}, \quad \al_{nk1} = \tilde \al_{nk}, \quad
    \al'_{nk0} = \al'_{nk}, \quad \al'_{nk1} = \tilde \al'_{nk}, \\
    S_{nks}(x) = S(x, \la_{nks}), \quad \tilde S_{nks}(x) = \tilde S(x, \la_{nks}), \quad n \ge 1, \: k = \overline{1, m}, \: s = 0, 1.
\end{gather*}

We group the square roots $\{ \rho_{nks} \}_{n \ge 1, \, k = \overline{1, m}, \, s = 0, 1}$ of the eigenvalues into the collections:
\begin{equation} \label{defG}
G_1 := \{ \rho_{nks} \}_{n = \overline{1, n_0}, \, k = \overline{1, m}, \, s = 0, 1}, \quad
G_{2j} := \{ \rho_{n_0 + j, ks} \}_{k = \overline{1,p}, \, s = 0, 1}, \quad
G_{2j+1} := \{ \rho_{n_0 + j, ks} \}_{k = \overline{p + 1, m}, \, s = 0, 1},
\end{equation}
for $j \ge 1$. 
Each collection $G_n$ is a multiset, i.e. it can contain multiple values.
In view of the asymptotics~\eqref{asymptla}, one can choose $n_0 \ge 1$, such that $G_n \cap G_k = \varnothing$ for all $n \ne k$.

For any finite multiset $\mathcal G$ of complex numbers, we define the finite-dimensional space $B(\mathcal G)$. The space $B(\mathcal G)$ consists of all the matrix functions $f \colon \mathcal G \to \mathbb C^{m \times m}$ with the property: if $\rho, \theta \in \mathcal G$ and $\rho = \theta$, then $f(\rho) = f(\theta)$. The norm in $B(\mathcal G)$ is defined as follows:
$$
   \| f \|_{B(\mathcal G)} = \max \left\{ \max_{\rho \in \mathcal G} \| f(\rho) \|, \max_{\substack{\rho \ne \theta \\ \rho, \theta \in \mathcal G}} 
   |\rho - \theta|^{-1} \| f(\rho) - f(\theta) \| \right\}.
$$

Define the Banach space $B$ of infinite sequences:
\begin{equation} \label{defB}
B = \{ f = \{ f_n \}_{n \ge 1} \colon f_n \in B(G_n), \, n \ge 1, \, \| f \|_B < \iy \}, \quad \| f \|_B := \sup_{n \ge 1} (n \| f_n\|_{B(G_n)}).
\end{equation}
One can easily show, that the following sequence $\psi(x)$ belongs to $B$ for each fixed $x \in [0, \pi]$:
$$
\psi(x) = \{ \psi_n(x) \}_{n \ge 1}, \quad \psi_n(x)(\rho) = S(x, \rho^2), \quad \rho \in G_n, \quad n \ge 1.
$$
Analogously $\tilde \psi(x) \in B$ is defined, by changing $S(x, \rho^2)$ to $\tilde S(x, \rho^2)$.

For each fixed $x \in [0, \pi]$, we also define the linear operator $R(x) \colon B \to B$, acting on any element $f = \{ f_n \}_{n \ge 1}$ of $B$ as follows:
\begin{gather*}
    (f R(x))_n = \sum_{k = 1}^{\iy} f_k R_{k, n}(x), \quad R_{k, n}(x) \colon B(G_k) \to B(G_n), \\ 
    (f_k R_{k, n}(x))(\rho) = \sum_{(l, j) \colon \rho_{lj0}, \rho_{lj1} \in G_k} (f_k (\rho_{lj0}) \al'_{lj0} D(x, \rho^2_{lj0}, \rho^2) - 
    f_k(\rho_{lj1}) \al'_{lj1} D(x, \rho_{lj1}^2, \rho^2)), \quad \rho \in G_n.
\end{gather*}

In the latter expressions, the operators are put to the right of their operands to emphasize that the matrices are multiplied in this order. 
Similarly, the operator $\tilde R(x)$ is defined, by changing $S$, $D$ to $\tilde S$, $\tilde D$.
The operators $R(x)$ and $\tilde R(x)$ are compact on $B$. Furthermore, the following {\it main equation} is satisfied for each fixed $x \in [0, \pi]$:
\begin{equation} \label{main}
    \tilde \psi(x) = \psi(x) (\mathcal I + \tilde R(x)),
\end{equation}
where $\mathcal I$ is the identity operator in $B$. The main equation~\eqref{main} can be used for constructive solution of Inverse Problems~\ref{ip:graph} and~\ref{ip:matr} (see \cite{Bond19-alg} for details).

We call an element $f = \{ f_n \}_{n \ge 1} \in B$ {\it diagonal}, if for every $n \ge 1$, all the values of the matrix function $f_n \colon G_n \to \mathbb C^{m \times m}$ are diagonal matrices. In the graph case, the matrix function $S(x, \la)$ is diagonal, so the element $\psi(x) \in B$ is diagonal for all $x \in [0, \pi]$.

\section{Spectral Data Characterization}

In this section, we formulate necessary and sufficient conditions of solvability for Inverse Problems~\ref{ip:graph} and~\ref{ip:matr}.

Let $SD$ be the class of all the data in the form $\{ \la_{nk}, \, \al_{nk} \}_{n \ge 1, \, k = \overline{1, m}}$, such that: 
\begin{enumerate}
    \item $\la_{nk}$ are real numbers, $\al_{nk}$ are Hermitian, nonnegative definite matrices: $\al_{nk} = \al_{nk}^{\dagger} \ge 0$ for all $n \ge 1$, $k = \overline{1, m}$;
    \item if $\la_{n_1, k_1} = \la_{n_2, k_2}$, then $\al_{n_1, k_1} = \al_{n_2, k_2}$. Moreover, $\rank (\al_{nk})$ equals the multiplicity of the corresponding value $\la_{nk}$ (i.e. the number of times that $\la_{nk}$ occurs in the sequence).
\end{enumerate}

\begin{df} \label{def:F}
Following Definition~\ref{def:al}, consider any group of multiple eigenvalues $\la_{n_1, k_1} = \la_{n_2, k_2} = \dots = \la_{n_r, k_r}$ maximal by inclusion. We have $\rank(\al_{n_1, k_1}) = r$. In other words, $\Ran(\al_{n_1, k_1})$ is an $r$-dimensional subspace in $\mathbb C^m$. Choose in $\Ran(\al_{n_1, k_1})$ an orthonormal basis $\{ \mathcal E_{n_j, k_j} \}_{j = \overline{1, m}}$. (This choice can be not unique). Thus, the sequence of normalized vectors $\{ \mathcal E_{nk} \}_{n \ge 1, \, k = \overline{1, m}}$ is defined. Further we need the following sequence of vector functions
$$
\mathscr F := \left\{ \mathcal E_{nk} \frac{\sin (\rho_{nk} t)}{\rho_{nk}}\right\}_{n \ge 1, \, k = \overline{1, m}}.
$$
\end{df}

The results of \cite{Bond19-sp, Bond19-asympt, Bond19-alg} yield the following necessary conditions on the spectral data.

\begin{prop}[Necessity] \label{prop:nc}
The spectral data $\mathcal S := \{ \la_{nk}, \al_{nk} \}_{n \ge 1, \, k = \overline{1, m}}$ of the problem $L$ in the general form fulfill the following conditions.
\begin{enumerate}
    \item $\mathcal S \in SD$.
    \item \textsc{(Asymptotics)} The eigenvalues~$\{ \la_{nk} \}_{n \ge 1, \, k = \overline{1, m}}$ and the weight matrices $\{ \al_{nk} \}_{n \ge 1, \, k = \overline{1, m}}$ satisfy the asymptotic relations~\eqref{asymptla} and~\eqref{asymptal}, respectively, where $\{ z_k \}_{k = 1}^m$ are the roots of the polynomials $\mathcal P_1(z)$ and $\mathcal P_2(z)$ defined by~\eqref{defP12}, and $\{ A^{(s)} \}_{s = 1}^m$ are the matrices defined in Proposition~\ref{prop:asymptal}.
    \item \textsc{(Completeness)} The sequence $\mathscr F$ is complete in $L_2((0, \pi); \mathbb C^{m})$ for any choice of the vectors $\{ \mathcal E_{nk} \}_{n \ge 1, \, k = \overline{1, m}}$ in Definition~\ref{def:F}.
    \item \textsc{(Solvability)} The main equation \eqref{main} is uniquely solvable for each $x \in [0, \pi]$.
\end{enumerate}

In the graph case, the solution of the main equation~\eqref{main} is diagonal \textsc{(Diagonality)}.
\end{prop}

The main goal of this paper is to show that the conditions of Proposition~\ref{prop:nc} are not only necessary but also sufficient for solvability of Inverse Problems~\ref{ip:graph} and~\ref{ip:matr}. The main results are formulated as follows.

\begin{thm}[Sufficiency in the general case] \label{thm:sc}
Let $\mathcal S = \{ \la_{nk}, \al_{nk} \}_{n \ge 1, \, k = \overline{1, m}}$ be an arbitrary element of $SD$, satifying the following conditions.
\begin{enumerate}
    \item \textsc{(Asymptotics)} 
    The values $\{ \la_{nk} \}_{n \ge 1, \, k = \overline{1, m}}$ and the matrices $\{ \al_{nk} \}_{n \ge 1, \, k = \overline{1, m}}$ satisfy the relations~\eqref{asymptla} and~\eqref{asymptal}, respectively, where $T$ is an orthogonal projector in $\mathbb C^m$, $p := \rank(T) \in [1, m-1]$, $T^{\perp} = I - T$, $\{ z_k \}_{k = 1}^m$ are real numbers, $z_k \le z_{k + 1}$ for $k = \overline{1, m-1}\backslash \{ p \}$, $\{ A^{(s)} \}_{s = 1}^m$ are orthogonal projectors in $\mathbb C^m$, having the following properties:
\begin{gather*}
\sum_{s = 1}^p A^{(s)} = T, \quad \sum_{s = p + 1}^m A^{(s)} = T^{\perp}, \\
\rank(A^{(s)}) = \# \{ k = \overline{1, p} \colon z_k = z_s\}, \: s = \overline{1, p}, \\
\rank(A^{(s)}) = \# \{ k = \overline{p + 1, m} \colon z_k = z_s \}, \: s = \overline{p + 1, m}, \\
A^{(s)} A^{(k)} = 0, \quad s, k = \overline{1, m} \colon \: (s \le p \:\: \text{and} \:\: k > p) \:\: \text{or} \:\: (z_k \ne z_s).
\end{gather*}
    \item \textsc{(Completeness)} The vectors $\{ \mathcal E_{nk} \}_{n \ge 1, \, k = \overline{1, m}}$ in Definition~\ref{def:F} can be chosen so that the sequence $\mathscr F$ is complete in $L_2((0, \pi); \mathbb C^{m})$.
\end{enumerate}
Then there exists a unique boundary value problem $L(Q(x), T, H)$ in the general form, such that $\mathcal S$ is its spectral data.
\end{thm}

\begin{thm}[Sufficiency in the graph case] \label{thm:scg}
Let $\mathcal S = \{ \la_{nk}, \al_{nk} \}_{n \ge 1, \, k = \overline{1, m}}$ be an arbitrary element of $SD$, satisfying the following conditions.
\begin{enumerate}
    \item \textsc{(Asymptotics)} There exist real numbers $\{ \om_j \}_{j = 1}^m$, such that for the values $\{ \la_{nk} \}_{n \ge 1, \, k = \overline{1, m}}$ and for the matrices $\{ \al_{nk} \}_{n \ge 1, \, k = \overline{1,m }}$ the asymptotic relations~\eqref{asymptla} and~\eqref{asymptal} hold, respectively, where $p = 1$, $T$ and $T^{\perp}$ are the matrices defined by~\eqref{defT}, $z_1$ is an arbitrary real number, $\{ z_k \}_{k = 2}^m$ are the roots of the polynomial $\mathcal P_2(z)$ defined by~\eqref{defP2}, $z_k \le z_{k+1}$, $k = \overline{2, m-1}$, the matrices $\{ A^{(s)} \}_{s = 1}^m$ are defined by~\eqref{defAs}, \eqref{defA}.
    \item \textsc{(Completeness)} The vectors $\{ \mathcal E_{nk} \}_{n \ge 1, \, k = \overline{1, m}}$ in Definition~\ref{def:F} can be chosen so that the sequence $\mathscr F$ is complete in $L_2((0, \pi); \mathbb C^{m})$.
\end{enumerate}

Let $\tilde L = L(\tilde Q(x), \tilde T, \tilde H)$ be constructed by the formulas~\eqref{modelg}, where $\Omega := \diag\{ \om_j \}_{j = 1}^m$.
Under the conditions~1-2, the main equation~\eqref{main} is uniquely solvable \textsc{(Solvability)}.

If, in addition, the solution $\psi(x)$ of the main equation is diagonal \textsc{(Diagonality)}, then there exist unique real-valued functions $\{ q_j \}_{j = 1}^m$, $q_j \in L_2(0, \pi)$, $j = \overline{1, m}$, such that $\mathcal S$ is the spectral data of the operator $\mathcal L$, constructed by $\{ q_j \}_{j = 1}^m$ and $h = \tilde h$ ($\tilde h$ is defined by~\eqref{modelg}).

\end{thm}

Thus, Proposition~\ref{prop:nc} and Theorem~\ref{thm:sc} together give characterization of the spectral data for the matrix Sturm-Liouville problem \eqref{eqv}-\eqref{bc} in the general form. Proposition~\ref{prop:nc} together with Theorem~\ref{thm:scg} characterize spectral data for the Sturm-Liouville operator $\mathcal L$ on the star-shaped graph. Note that, in Theorem~\ref{thm:scg}, \textsc{Solvability} of the main equation is not required, but it follows from \textsc{Asymptotics} and \textsc{Completeness}.

Theorems~\ref{thm:sc} and~\ref{thm:scg} are proved in Sections 4-6.

\section{Estimates}

This section plays an auxiliary role. Here we obtain asymptotic formulas and estimates used in the further proofs.

Recall the notation $\tau := \mbox{Im}\,\rho$. The matrix solution $S(x, \la)$ has the following standard asymptotics as $|\rho| \to \iy$, $\la = \rho^2$:
\begin{equation} \label{asymptS}
    S(x, \la) = \frac{\sin \rho x}{\rho} I + O\left( \rho^{-2} \exp(|\tau| x)\right), \quad
    S'(x, \la) = \cos \rho x I + O\left( \rho^{-1} \exp(|\tau| x)\right).
\end{equation}


Let $\Psi(x, \la)$ be the matrix solution of~\eqref{eqv} under the initial conditions $\Psi(\pi, \la) = T$, $\Psi'(\pi, \la) = T^{\perp} + H T$. Clearly, $V(\Psi) = 0$. The following asymptotic formulas are valid as $|\rho| \to \iy$:
\begin{align*}
    \Psi(x, \la) & = \Bigl(\cos \rho (\pi - x) I + O\left( \rho^{-1} \exp(|\tau|(\pi - x))\right)\Bigr) T \\ & + \Bigl( -\frac{\sin \rho (\pi - x)}{\rho} I + O\left( \rho^{-2} \exp(|\tau|(\pi - x))\right)\Bigr) T^{\perp}, \\
    \Psi'(x, \la) & = \Bigl( \rho \sin \rho (\pi - x) I + O\left( \exp(|\tau|(\pi - x))\right)\Bigr) T \\ & + \Bigl( \cos \rho(\pi - x) I + O\left( \rho^{-1} \exp(|\tau|(\pi - x))\right)\Bigr) T^{\perp}. 
\end{align*}

The Weyl solution can be expressed in the form
$$
\Phi(x, \la) = \Psi(x, \la) \Psi^{-1}(0, \la).
$$
For $|\rho| \to \iy$, $\rho \in G_{\de}$, where
$$
G_{\de} := \{ \rho \in \mathbb C \colon |\rho - n| \ge \de, \, |\rho - (n - \tfrac{1}{2})| \ge \de \}, \quad \de > 0,
$$
we have
$$
\Psi^{-1}(0, \la) = T \biggl( \frac{1}{\cos \rho \pi} I + O\left( \rho^{-1} \exp(-|\tau|\pi)\right)\biggr) -T^{\perp} \biggl(\frac{\rho}{\sin \rho \pi} + O\left( \exp(-|\tau|\pi)\right) \biggr).
$$
Consequently, we obtain the asymptotic relations
\begin{align} \label{asymptPhi}
    & \Phi(x, \la) = T \frac{\cos \rho (\pi - x)}{\cos \rho \pi} + T^{\perp} \frac{\sin \rho (\pi - x)}{\sin \rho \pi} + O\left( \rho^{-1} \exp(-|\tau| x)\right), \\ \label{asymptPhip}
    & \Phi'(x, \la) = T \frac{\rho \sin \rho(\pi - x)}{\cos \rho \pi} - T^{\perp} \frac{\rho \cos \rho(\pi - x)}{\sin \rho \pi} + O\left( \exp(-|\tau| x)\right),
\end{align}
for $|\rho| \to \iy$, $\rho \in G_{\de}$.

Define the matrix function
\begin{equation} \label{defE}
    E(x, \la, \mu) := \frac{\langle S^{\dagger}(x, \bar \la), \Phi(x, \mu) \rangle}{\la - \mu}.
\end{equation}
Using the asymptotic formulas~\eqref{asymptS}, \eqref{asymptPhi} and \eqref{asymptPhip} together with the definitions~\eqref{defD} and~\eqref{defE}, we obtain the estimates
\begin{align} \label{estD}
    & \| D(x, \rho^2, \theta^2) \| \le \frac{C \exp((|\mbox{Im}\,\rho| + |\mbox{Im}\,\theta|) x)}{(|\rho| + 1)(|\theta| + 1)(|\rho - \theta| + 1)}, \quad \rho, \theta \in \mathbb C,  \\ \label{estE}
    & \| E(x, \rho^2, \theta^2) \| \le \frac{C \exp((|\mbox{Im}\, \rho| - |\mbox{Im}\, \theta|) x)}{(|\rho| + 1)(|\rho - \theta| + 1)}, \quad \rho \in \mathbb C, \quad
    \theta \in G_{\de}, 
\end{align}
where $x \in [0, \pi]$, and $C$ is a positive constant independent of $x$, $\rho$ and $\theta$.

Along with the problem $L$, consider a problem $\tilde L$, such that 
\begin{equation} \label{eqcoef}
p = \tilde p, \quad T = \tilde T, \quad z_s = \tilde z_s, \quad A^{(s)} = \tilde A^{(s)}, \quad s = \overline{1, m},
\end{equation}
i.e. all the coefficients in the asymptotic formulas~\eqref{asymptla} and~\eqref{asymptal} for the problems $L$ and $\tilde L$ coincide.

Consider the collections~$\{ G_k \}_{k \ge 1}$ defined by~\eqref{defG}. Introduce the notations
$$
n_1 := 0, \quad n_{2j} = n_0 + j - \frac{1}{2}, \quad n_{2j + 1} = n_0 + j, \quad j \ge 1,
$$
i.e. $n_k$ is the main part in the asymptotic relations~\eqref{asymptla} for the values from $G_k$. 
    
Denote by $r_1$ and $r_2$ the numbers of distinct values among $\{ z_k \}_{k = 1}^p$ and among $\{ z_k \}_{k = p + 1}^m$, respectively. Consider the index sets
$$
\mathcal J_s := \begin{cases}
\{ k = \overline{1, p} \colon z_k = z_s \}, \quad s = \overline{1, p}, \\
\{ k = \overline{p+1, m} \colon z_k = z_s \}, \quad s = \overline{p + 1, m}.
\end{cases}
$$
Denote all the distinct sets from $\{ \mathcal J_s \}_{s = 1}^p$ and $\{ \mathcal J_s \}_{s = p + 1}^m$ by $\{ J^{(1)}_s \}_{s = 1}^{r_1}$ and $\{ J^{(2)}_s \}_{s = 1}^{r_2}$, respectively.

We divide every collection $G_k$ into subcollections as follows:
\begin{gather*}
G_k = \bigcup_{i = 1}^{p_k} G_{ki}, \quad
p_1 := 1, \quad p_{2j} := r_1, \quad p_{2j + 1} := r_2, \\
G_{11} := G_1, \quad G_{2j,i} := \{ \rho_{n_0 + j, ks} \}_{k \in J_i^{(1)}, \, s = 0, 1}, \quad
G_{2j+1, i} := \{ \rho_{n_0 + j, ks} \}_{k \in J_i^{(2)}, \, s = 0, 1}.
\end{gather*}
In $n_0$ is sufficiently large, we have $G_{ki} \cap G_{kj} = \varnothing$, $i \ne j$, $k \ge 1$.

For any collection $\mathcal G$, introduce the notations
$$
\al(\mathcal G) = \sum_{(l, j) \colon \rho_{lj0} \in \mathcal G} \al'_{lj0}, \quad \tilde \al(\mathcal G) = \sum_{(l, j) \colon \rho_{lj1} \in \mathcal G} \al'_{lj1}.
$$

In view of the asymptotics~\eqref{asymptla}, \eqref{asymptal} and the relation~\eqref{eqcoef}, we have
\begin{gather} \label{defXi}
\Xi := \left( \sum_{k = 1}^{\iy} (k \xi_k)^2 \right)^{1/2} < \iy, \\ \label{defxi}
\xi_k := \sum_{i = 1}^{p_k} \sum_{\rho, \theta \in G_{ki}} |\rho - \theta| + \frac{1}{k^3} \sum_{i = 1}^{p_k} \| \al(G_{ki}) - \tilde \al(G_{ki}) \| + \frac{1}{k^2} \| \al(G_k) - \tilde \al(G_k) \|. 
\end{gather}

Using the estimates~\eqref{estD} and \eqref{estE} for $\tilde D$ and $\tilde E$, respectively, and the standard approach (see, e.g., \cite[Lemma~1.6.2]{FY01}), based on Schwarz's lemma, we obtain the following result.

\begin{lem} \label{lem:Schwarz}
For $x \in [0, \pi]$, $k \ge 1$, the following estimates are valid:
\begin{gather*}
    \| \tilde D(x, \chi^2, \rho^2) - \tilde D(x, \theta^2, \rho^2) \| \le \frac{C \exp(|\tau| x)}{k^2 (|\rho| + 1) (|\rho - n_k| + 1)}, \quad \rho \in \mathbb C, \quad \theta, \chi \in G_k, \\
\| \tilde D(x, \chi^2, \rho^2) - \tilde D(x, \theta^2, \rho^2) \| \le \frac{C \xi_k \exp(|\tau|x)}{k (|\rho| + 1) (|\rho - n_k| + 1)}, \quad \rho \in \mathbb C, \quad \chi, \theta \in G_{ki}, \quad i = \overline{1, p_k}, \\
\| \tilde E(x, \chi^2, \rho^2) - \tilde E(x, \theta^2, \rho^2) \| \le \frac{C \exp(-|\tau| x)}{k^2 (|\rho - n_k| + 1)}, \quad \rho \in G_{\de}, \quad \theta, \chi \in G_k, \\
\| \tilde E(x, \chi^2, \rho^2) - \tilde E(x, \theta^2, \rho^2) \| \le \frac{C \xi_k \exp(-|\tau|x)}{k (|\rho - n_k| + 1)}, \quad \rho \in G_{\de}, \quad \chi, \theta \in G_{ki}, \quad i = \overline{1, p_k},
\end{gather*}
where the constant $C$ does not depend on $x$, $\rho$, $k$, $\chi$ and $\theta$.
\end{lem}

The following proposition has been proved in \cite{Bond19-alg}.

\begin{prop} \label{prop:estR}
For $x \in [0, \pi]$, the following estimates hold:
\begin{gather*}
    \| \tilde R_{k, n}(x) \|_{B(G_k) \to B(G_n)} \le \frac{C k \xi_k}{n (|n - k| + 1)}, \quad n, k \ge 1, \\
    \| \tilde R(x) \|_{B \to B} \le C \Xi.
\end{gather*}
\end{prop}

\section{Main Equation Solvability}

The aim of this section is to prove that \textsc{Asymptotics} and \textsc{Completeness} conditions of Theorems~\ref{thm:sc} and~\ref{thm:scg} imply the unique solvability of the main equation~\eqref{main}.

Let $\mathcal S := \{ \la_{nk}, \al_{nk} \}_{n \ge 1, \, k = \overline{1, m}}$ be data from the class $SD$, satisfying the conditions of Theorem~\ref{thm:sc}. Then the integer $p \in [1, m-1]$, the reals $\{ z_k \}_{k = 1}^m$ and the matrices $\{ A^{(s)} \}_{s = 1}^m$ are specified by the asymptotic condition. Construct the matrix
$$
\tilde \Theta := \sum_{s \in J} z_s A^{(s)}, \quad
 J := \{ s = \overline{1, p} \colon s = 1 \:\: \text{or} \:\: s = p+1 \:\: \text{or} \:\: z_s \ne z_{s - 1} \}.
$$

Put $\tilde L = L(\frac{2}{\pi} \tilde \Theta, T, 0)$, where $T$ is the orthogonal projector from the asymptotics~\eqref{asymptal}. It is easy to check, that the spectral data $\{ \tilde \la_{nk}, \tilde \al_{nk} \}_{n \ge 1, \, k = \overline{1, m}}$ of the problem $\tilde L$ satisfy the asymptotic relations~\eqref{asymptla} and~\eqref{asymptal} with 
the same coefficients $\{ z_k \}_{k = 1}^m$, $T$, $T^{\perp}$ and $\{ A^{(s)} \}_{s = 1}^m$ as the collection $\mathcal S$ has.
Consequently, the estimates of Section~4 are valid for the problems $L$ and $\tilde L$.
The results of~\cite{Bond19-alg} yield that the operator $\tilde R(x)$, constructed in Section~2, is compact in $B$. Relying on these facts, we prove the following lemma.

\begin{lem} \label{lem:main}
Under the above assumptions, the main equation~\eqref{main} has the unique solution in $B$ for each $x \in [0, \pi]$.
\end{lem}

\begin{proof}
Fix $x \in [0, \pi]$. Let us prove that the operator $(I + \tilde R(x))$ has a bounded inverse. By virtue of Fredholm Theorem, it is sufficient to show that the homogeneous equation $f (I + \tilde R(x)) = 0$ has the only solution $f = 0$ in $B$.
Due to the definitions in Section~2, a solution $f = \{ f_k \}_{k \ge 1} \in B$ of the homogeneous equation satisfies the relations
$$
f_n(\rho) + \sum_{k = 1}^{\iy} \sum_{(l, j) \colon \rho_{lj0}, \rho_{lj1} \in G_k} (f_k(\rho_{lj0}) \al'_{lj0} \tilde D(x, \rho_{lj0}^2, \rho^2) - f_k(\rho_{lj1}) \al'_{lj1} \tilde D(x, \rho_{lj1}^2, \rho^2)) = 0, \quad \rho \in G_n,
$$
and the estimates
\begin{equation} \label{estf}
\|f_k(\rho)\| \le \frac{\| f \|_B}{k}, \quad \| f_k(\rho) - f_k(\theta) \| \le \frac{\| f \|_B |\rho - \theta|}{k}, \quad \rho, \theta \in G_k, \quad k \ge 1.
\end{equation}

Introduce the matrix functions 
\begin{gather*}
    \ga(\la) := -\sum_{k = 1}^{\iy} \sum_{(l, j) \colon \rho_{lj0}, \rho_{lj1} \in G_k} (f_k(\rho_{lj0}) \al'_{lj0} \tilde D(x, \la_{lj0}, \la) - f_k(\rho_{lj1}) \al'_{lj1} \tilde D(x, \la_{lj1}, \la)), \\
    \Gamma(\la) := -\sum_{k = 1}^{\iy} \sum_{(l, j) \colon \rho_{lj0}, \rho_{lj1} \in G_k} (f_k(\rho_{lj0}) \al'_{lj0} \tilde E(x, \la_{lj0}, \la) - f_k(\rho_{lj1}) \al'_{lj1} \tilde E(x, \la_{lj1}, \la)), \\
    F(\la) := \Gamma(\la) \ga^{\dagger}(\bar \la).
\end{gather*}

The matrix function $\ga(\la)$ is entire and $\ga(\la_{ljs}) = f_k(\rho_{ljs})$ for $\rho_{ljs} \in G_k$, $k \ge 1$. The matrix functions $\Gamma(\la)$ and $F(\la)$ are meromorphic with simple poles from the set $\{ \la_{ljs} \}_{l \ge 1, \, j = \overline{1, m}, \, s = 0, 1}$. Calculations yield
\begin{equation} \label{ResF}
\Res_{\la = \la_{lj0}} F(\la) = \ga(\la_{lj0}) \al_{lj0} \ga{\dagger}(\la_{lj0}), \quad
\Res_{\la = \la_{lj1}} F(\la) = 0, \quad
l \ge 1, \quad j = \overline{1, m},
\end{equation}
if $\la_{lj1} \not\in \{ \la_{nk} \}_{n \ge 1, \, k = \overline{1, m}}$. The opposite case requires minor technical modifications.

Using the relations~\eqref{estD}, \eqref{estE}, \eqref{defxi} and Lemma~\ref{lem:Schwarz}, we obtain the estimates
\begin{align} \label{estga}
    & \| \ga(\la) \| \le \frac{C \exp(|\tau|x)}{|\rho| + 1} \sum_{k = 1}^{\iy} \frac{\xi_k}{|\rho - n_k| + 1}, \quad \rho \in \mathbb C, \\ \label{estGa}
    & \| \Gamma(\la) \| \le C \exp(-|\tau|x) \sum_{k = 1}^{\iy} \frac{\xi_k}{|\rho - n_k| + 1}, \quad \rho \in G_{\de}.
\end{align}

Consider the contours $\Gamma_N := \{ \la \in \mathbb C \colon |\la| = (N + \tfrac{1}{4})^2 \}$, $N \in \mathbb N$, in the $\la$-plane with the counter-clockwise circuit. Clearly, $\la \in \Gamma_N$ implies $\rho = \sqrt \la \in G_{\de}$ for sufficiently large $N$ and sufficiently small $\de > 0$. Using the estimates~\eqref{estga} and~\eqref{estGa}, we obtain
$$
\| F(\la) \| \le \frac{C}{|\rho|^3}, \quad \la \in \Gamma_N.
$$
Hence, on the one hand, we have
$$
\lim_{N \to \iy} \oint_{\Gamma_N} F(\la) \, d\la = 0.
$$
On the other hand, Residue Theorem together with the relations~\eqref{ResF} imply
$$
\frac{1}{2 \pi i} \oint_{\Gamma_N} F(\la) \, d\la =  \sum_{(l, j) \colon |\la_{lj0}| < \left(N + \tfrac{1}{4}\right)^2} \ga(\la_{lj0}) \al'_{lj0} \ga^{\dagger}(\la_{lj0}).
$$
Taking the limit as $N \to \iy$, we arrive at the relation
$$
\sum_{l = 1}^{\iy} \sum_{j = 1}^m \ga(\la_{lj0}) \al'_{lj0} \ga^{\dagger}(\la_{lj0}) = 0.
$$
Since $\al_{lj0} = \al_{lj0}^{\dagger} \ge 0$ for all $l \ge 1$, $j = \overline{1, m}$, we conclude that
\begin{equation} \label{relga}
\ga(\la_{lj0}) \al_{lj0} = 0, \quad l \ge 1, \quad j = \overline{1, m}.
\end{equation}

Note that the function $\rho \ga(\rho^2)$ is entire and odd in the $\rho$-plane. In view of~\eqref{defXi} and \eqref{estga}, we have
$$
\| \ga(\la) \| = O\left( \rho^{-1} \exp(|\tau|\pi) \right), \quad |\rho| \to \iy, \qquad \rho \ga(\rho^2) \in L_2(\mathbb R; \mathbb C^{m \times m}).
$$
Therefore Paley-Wiener Theorem yields the representation
$$
\ga(\la) = \int_0^{\pi} \mathscr G(t) \frac{\sin \rho t}{\rho} \, dt, \quad \mathscr G \in L_2((0, \pi); \mathbb C^{m \times m}).
$$
The relation~\eqref{relga} implies that
$$
\ga(\la_{nk}) \mathcal E_{nk} = \int_0^{\pi} \mathscr G(t) \frac{\sin \rho_{nk} t}{\rho_{nk}} \mathcal E_{nk} \, dt = 0, \quad n \ge 1, \quad k = \overline{1, m},
$$
where $\{ \mathcal E_{nk} \}$ are the vectors from Definition~\ref{def:F}. Since the sequence $\mathscr F$ is complete in $L_2((0, \pi); \mathbb C^m)$, we conclude that $\mathscr G = 0$ in $L_2((0, \pi); \mathbb C^{m \times m})$. Consequently, $\ga(\la) \equiv 0$. Hence $f = 0$ in $B$, which yields the claim of the lemma.
\end{proof}

\begin{cor}
Suppose that data $\mathcal S \in SD$ satisfy the conditions 1-2 of Theorem~\ref{thm:scg}. Let $\tilde L$ be the auxiliary problem constructed by the formulas~\eqref{modelg}. Then the main equation~\eqref{main} is uniquely solvable.
\end{cor}

\section{Proof of Sufficiency}

In this section, several lemmas are provided, which finish the proofs of Theorems~\ref{thm:sc} and~\ref{thm:scg}. By using the solution $\psi(x)$ of the main equation~\eqref{main}, we construct $Q \in L_2((0, \pi); \mathbb C^{m \times m})$ and $H \in \mathbb C^{m \times m}$. Further we prove that the given collection $\mathcal S$ is the spectral data of the boundary value problem $L = L(Q(x), T, H)$.

Let data $\mathcal S = \{ \la_{nk}, \al_{nk} \}_{n \ge 1, \, k = \overline{1, m}}$ fulfill the conditions of Theorem~\ref{thm:sc}, and let $\tilde L$ be the model problem constructed in the previous section. By Lemma~\ref{lem:main} the main equation~\eqref{main}, constructed by $\mathcal S$ and $\tilde L$, has the unique solution $\psi(x) = \{ \psi_n(x) \}_{n \ge 1} \in B$ for each fixed $x \in [0, \pi]$.
Relying on~\eqref{defXi}, \eqref{defxi} and~\eqref{defB}, we prove the following lemma.

\begin{lem} \label{lem:estpsi}
For $n \ge 1$, the operator functions $\psi_n \colon [0, \pi] \to B(G_n)$ are continuously differentiable with respect to $x \in [0, \pi]$ and satisfy the estimates
\begin{gather} \label{estpsin1}
    \| \psi_n^{(\nu)}(x) \|_{B(G_n)} \le C n^{\nu - 1}, \quad \nu = 0, 1, \quad x \in [0, \pi], \\ \label{estpsin2}
    \| \psi_n(x) - \tilde \psi_n(x) \|_{B(G_n)} \le \frac{C \Xi \eta_n}{n}, \quad
    \| \psi'_n(x) - \tilde \psi'_n(x) \|_{B(G_n)} \le \frac{C \Xi}{n}, \quad x \in [0, \pi], \\ \label{defeta}
    \eta_n := \left(\sum_{k = 1}^{\iy} \frac{1}{k^2 (|n - k| + 1)^2} \right)^{1/2}, \quad \{ \eta_n \}_{n \ge 1} \in l_2.
\end{gather}
\end{lem}

\begin{proof}
Analogously to \cite[Lemma~4.3]{Bond19-alg}, we prove that, for every $k, n \ge 1$, the operator function $\tilde R_{k,n}(x)$ is continuously differentiable by $x$, and
\begin{equation} \label{difR}
\| \tilde R'_{k, n}(x) \|_{B(G_n) \to B(G_n)} \le \frac{C k \xi_k}{n}, \quad x \in [0, \pi],
\end{equation}
Fix $x_0 \in [0, \pi]$. By using~\eqref{difR} and~\eqref{defXi}, one can easily show that 
\begin{equation} \label{Rx0}
\| \tilde R(x) - \tilde R(x_0) \|_{B \to B} \le C \Xi |x - x_0|, \quad x \in [0, \pi].
\end{equation}

Define the operator $P(x) := (\mathcal I + \tilde R(x))^{-1}$. Relying on the estimate~\eqref{Rx0}, we prove that $P(x)$ is continuous on $[0, \pi]$. Consequently, $\| P(x) \|_{B \to B} \le C$, $x \in [0, \pi]$. Define $R(x) := \mathcal I - P(x)$, $R(x) = [R_{k,n}(x)]_{k, n \ge 1}$. Clearly, $\| R(x) \|_{B \to B} \le C$, $x \in [0, \pi]$, and
\begin{equation} \label{RR}
(\mathcal I - R(x)) (\mathcal I + \tilde R(x)) = (\mathcal I + \tilde R(x)) (\mathcal I - R(x)) = \mathcal I.
\end{equation}
The latter relations can be rewritten in the form
\begin{align} \label{Rkn1}
    & R_{k, n}(x) = \tilde R_{k, n}(x) - \sum_{l = 1}^{\iy} R_{k, l}(x) \tilde R_{l, n}(x), \\ \label{Rkn2}
    & R_{k, n}(x) = \tilde R_{k, n}(x) - \sum_{l = 1}^{\iy} \tilde R_{k, l}(x) R_{l, n}(x),
\end{align}
where $n, k \ge 1$. Using~\eqref{Rkn2}, Proposition~\ref{prop:estR} and the estimate $\| R(x) \|_{B \to B} \le C$, we get 
\begin{equation} \label{Rkn3}
    \| R_{k, n}(x) \|_{B(G_k) \to B(G_n)} \le \frac{C k \xi_k}{n}, \quad k, n \ge 1, \quad x \in [0, \pi].
\end{equation}
Next, using~\eqref{Rkn3} together with~\eqref{Rkn1} and Proposition~\ref{prop:estR}, we arrive at the estimate
\begin{equation} \label{Rkn4}
    \| R_{k, n}(x) \|_{B(G_k) \to B(G_n)} \le \frac{C k \xi_k}{n} \left( \frac{1}{|n - k| + 1} + \Xi \eta_n \right), 
\end{equation}
where $n, k \ge 1$, $x \in [0, \pi]$, and $\eta_n$ is defined in \eqref{defeta}.

Since $\psi(x) = \tilde \psi(x) (\mathcal I - R(x))$, we have
\begin{equation} \label{relpsi}
    \psi_n(x) = \tilde \psi_n(x) - \sum_{k = 1}^{\iy} \tilde \psi_k(x) R_{k, n}(x), \quad n \ge 1.
\end{equation}
Using~\eqref{Rkn4}, \eqref{relpsi} and the estimate $\| \tilde \psi_n(x) \|_{B(G_n)} \le \frac{C}{n}$, $n \ge 1$, we obtain~\eqref{estpsin1} for $\nu = 0$ and~\eqref{estpsin2}.

One can similarly prove~\eqref{estpsin1} for $\nu = 1$, differentiating the relation~\eqref{relpsi}. The necessary estimate for $\| R'_{k, n}(x) \|_{B(G_k) \to B(G_n)}$ can be obtained by differentiating~\eqref{RR}.
\end{proof}

Define the matrix functions $S_{ljs}(x) := \psi_k(x)(\rho_{ljs})$ for $l \ge 1$, $j = \overline{1, m}$, $s = 0, 1$, where $k$ is such that $\rho_{ljs} \in G_k$. Also define
\begin{align} \label{defS}
    & S(x, \la) := \tilde S(x, \la) - \sum_{l = 1}^{\iy} \sum_{j = 1}^m (S_{lj0}(x) \al'_{lj0} \tilde D(x, \la_{lj0}, \la) - S_{lj1}(x) \al'_{lj1} \tilde D(x, \la_{lj1}, \la), \\ \label{defPhi}
    & \Phi(x, \la) := \tilde \Phi(x, \la) - \sum_{l = 1}^{\iy} \sum_{j = 1}^m (S_{lj0}(x) \al'_{lj0} \tilde E(x, \la_{lj0}, \la) - S_{lj1}(x) \al'_{lj1} \tilde E(x, \la_{lj1}, \la)), \\ \label{defeps}
    & \eps_0(x) := \sum_{l = 1}^{\iy} \sum_{j = 1}^m (S_{lj0}(x) \al'_{lj0} \tilde S^{\dagger}_{lj0}(x) - S_{lj1}(x) \al'_{lj1} \tilde S^{\dagger}_{lj1}(x)), \quad \eps(x) := -2 \eps_0'(x).
\end{align}

Using the relations~\eqref{main} and~\eqref{defS}, one can easily show that
\begin{equation} \label{relSlj0}
    S(x, \la_{lj0}) = S_{lj0}(x), \quad l \ge 1, \quad j = \overline{1, m}.
\end{equation}

Following the algorithm for solving Inverse Problem~\ref{ip:matr} from \cite{Bond19-alg}, we find 
\begin{equation} \label{defQH}
    Q(x) := \tilde Q(x) + \eps(x), \quad H := \tilde H - T \eps_0(\pi) T.
\end{equation}

Using Lemma~\ref{lem:estpsi}, \eqref{asymptS}, \eqref{defXi} and~\eqref{defxi}, we obtain the following result.

\begin{lem} \label{lem:eps}
The series in~\eqref{defeps} converges uniformly with respect to $x \in [0, \pi]$ to an absolutely continuous matrix functions. Moreover, $\eps(x) \in L_2((0, \pi); \mathbb C^{m \times m})$ and
$\| \eps \|_{L_2} \le C \Xi$.
Hence $Q(x)$ defined by~\eqref{defQH} belongs to $L_2((0, \pi); \mathbb C^{m \times m})$.
\end{lem}

Put $L = L(Q(x), T, H)$, where $Q(x)$ and $H$ are constructed by~\eqref{defQH}.
Our next goal is to show that $S(x, \la)$ is the sine-type matrix solution of eq.~\eqref{eqv} with the matrix potential $Q(x)$, and that $\Phi(x, \la)$ is the Weyl solution of the problem $L$.

\begin{lem} \label{lem:relSPhi}
The following relations hold
\begin{gather} \label{eqvS}
    -S''(x, \la) + Q(x) S(x, \la) = \la S(x, \la), \quad
    -\Phi''(x, \la) + Q(x) \Phi(x, \la) = \la \Phi(x, \la), \\ \label{icS}
    S(0, \la) = 0, \quad S'(0, \la) = I, \quad \Phi(0, \la) = I.
\end{gather}
\end{lem}

The relations~\eqref{eqvS} can be proved by differentiating~\eqref{defS} and~\eqref{defPhi}, analogously to the scalar case (see \cite[Lemma~1.6.9]{FY01}). The relations~\eqref{icS} trivially follow from~\eqref{defS} and~\eqref{defPhi}.

\begin{lem} \label{lem:fL2}
The following series converges in $L_2((0, \pi); \mathbb C^{m \times m})$:
\begin{equation} \label{deff}
f(x) = \sum_{n = 1}^{\iy} \sum_{k = 1}^m V(S(x, \la_{nk})) \al'_{nk} \tilde S^{\dagger}(x, \la_{nk}).
\end{equation}
\end{lem}

\begin{proof}
Note that
\begin{align*}
& \sum_{k = 1}^p V(S(x, \la_{nk})) \al'_{nk} \tilde S^{\dagger}(x, \la_{nk}) = Z_{n1}(x) + Z_{n2}(x) + Z_{n3}(x), \\
& Z_{n1}(x) := \sum_{s \in \mathscr S} \sum_{\substack{k = 1 \\ z_k = z_s}}^p (V(S(x, \la_{nk})) - V(S(x, \la_{ns})) \al'_{nk} \tilde S^{\dagger}(x, \la_{nk}) \\ 
& Z_{n2}(x) := \sum_{s \in \mathscr S} V(S(x, \la_{ns})) \sum_{\substack{k = 1 \\ z_k = z_s}}^p \al'_{nk} (\tilde S^{\dagger}(x, \la_{nk}) - \tilde S^{\dagger}(x, \la_{ns})) \\ 
& Z_{n3}(x) := \sum_{s \in \mathscr S} V(S(x, \la_{ns})) \al_n^{(s)} \tilde S^{\dagger}(x, \la_{ns}),
\end{align*}
where $\mathscr S := \{ s = \overline{1, p} \colon s = 1 \: \text{or} \: z_{s - 1} \ne z_s \}$.

Since the function $S(x, \la)$ satisfies~\eqref{eqvS} with $Q \in L_2((0, \pi); \mathbb C^{m \times m})$, the asymptotic relations~\eqref{asymptS} hold. Using \eqref{asymptS} for $S(x, \la)$ and $\tilde S(x, \la)$ together with~\eqref{asymptla}, we obtain for $n \ge 1$, $k = \overline{1, p}$ that
\begin{align} \label{asymptVS}
   & V(S(x, \la_{nk})) = \frac{(-1)^n}{n} (T (z_k I - \Omega + H) + T^{\perp} + K_n), \\ \label{asympttS}
    & \tilde S(x, \la_{nk}) = \frac{\sin \left(\bigl(n - \tfrac{1}{2}\bigr)  x \right) }{n - \tfrac{1}{2}} + \frac{1}{n^2} \cos \left(\bigl(n - \tfrac{1}{2} \bigr) x \right) \left( \frac{z_k x}{\pi} I - \frac{1}{2} \int_0^x \tilde Q(t) \, dt\right) + \frac{K_n}{n^2}.
\end{align}
Using~\eqref{asymptVS} and~\eqref{asympttS} and noting that $\al_{nk} = O(n^2)$, $n \ge 1$, one can easily show that the series $\sum\limits_{n = 1}^{\iy} Z_{ni}$ converges in $L_2((0, \pi); \mathbb C^{m \times m})$ for $i = 1, 2$. In view of \eqref{asymptVS}, \eqref{asympttS} and~\eqref{asymptal}, the following relation holds:
\begin{equation} \label{Zn3}
Z_{n3}(x) = \frac{2(-1)^n}{\pi} \sum_{s \in \mathscr S}  T (z_s I - \Omega + H) T A^{(s)} \sin \left( \bigl( n - \tfrac{1}{2} \bigr)x \right) + Z_{n4}(x), 
\end{equation}
where the series $\sum\limits_{n = 1}^{\iy} Z_{n4}$ converges in $L_2((0, \pi); \mathbb C^{m \times m})$. It follows from~\eqref{defQH} and construction of the model problem $\tilde L$, that 
$$
T( \Omega - H) T = T(\tilde \Omega - \tilde H) T = T \Theta T = T \sum_{s \in \mathscr S} z_s A^{(s)} T.
$$
Hence the main part in~\eqref{Zn3} vanishes, so $\sum\limits_{n = 1}^{\iy} Z_{n3}$ converges in $L_2((0, \pi); \mathbb C^{m \times m})$. The $L_2$-convergence of the series
$$
\sum_{n = 1}^{\iy} \sum_{k = p+1}^m V(S(x, \la_{nk})) \al'_{nk} \tilde S^{\dagger}(x, \la_{nk})
$$
can be proved analogously.
\end{proof}

The asymptotics~\eqref{asymptS} for $\tilde S(x, \la)$, the \textsc{Completeness} condition for $\mathcal S$ and $\tilde {\mathcal S}$ and the results of \cite{Bond19-sp} yield the following proposition.

\begin{prop} \label{prop:basis}
(i) The sequence $\{ \tilde S_{nk1}(x) \tilde {\mathcal E}_{nk} \}_{n \ge 1, \, k = \overline{1, m}}$ is complete in $L_2((0, \pi); \mathbb C^m)$.

(ii) The sequence $\{ \tilde S_{nk0}(x) \mathcal E_{nk} \}_{n \ge 1, \, k = \overline{1, m}}$ is minimal in $L_2((0, \pi); \mathbb C^m)$.
\end{prop}

\begin{lem} \label{lem:VPhi}
$V(\Phi) = 0$.
\end{lem}

\begin{proof}
Differentiating~\eqref{defS} and using~\eqref{defeps}, we derive
$$
V(S) = V(\tilde S) - \sum_{l = 1}^{\iy} \sum_{j = 1}^m (V(S_{lj0}) \al'_{lj0} \tilde D(\pi, \la_{lj0}, \la) - V(S_{lj1}) \al'_{lj1} \tilde D(\pi, \la_{lj1}, \la)) - T \eps_0(\pi) \tilde S(\pi, \la). 
$$
Since 
$$
\tilde V(\tilde S) = T (\tilde S'(\pi, \la) - \tilde H \tilde S(\pi, \la)) - T^{\perp} \tilde S(\pi, \la) = V(\tilde S) + T(H - \tilde H) \tilde S(\pi, \la),
$$
we get
\begin{equation} \label{VS1}
V(S) = \tilde V(\tilde S) - T \eps_0(\pi) T^{\perp} \tilde S(\pi, \la) - \sum_{l = 1}^{\iy} \sum_{j = 1}^m (V(S_{lj0}) \al'_{lj0} \tilde D(\pi, \la_{lj0}, \la) - V(S_{lj1}) \al'_{lj1} \tilde D(\pi, \la_{lj1}, \la)).
\end{equation}

Recall that $\{ \la_{nk1} \}_{n \ge 1, \, k = \overline{1, m}}$ and $\{ \al_{nk1} \}_{n \ge 1, \, k = \overline{1, m}}$ are the eigenvalues and the weight matrices of the problem $\tilde L$, respectively. Therefore 
\begin{equation} \label{VS2}
\tilde V(\tilde S(x, \la_{nk1})) \al_{nk1} = 0, \quad T^{\perp} \tilde S(\pi, \la_{nk1}) \al_{nk1} = 0, \quad n \ge 1, \quad k = \overline{1, m}
\end{equation}
(see \cite[Lemma~2.2]{Bond19-sp}). The relations~\eqref{VS1} and~\eqref{VS2} together imply
\begin{equation} \label{VS3}
V(S_{nk1}) \al_{nk1} = -\sum_{l = 1}^{\iy} \sum_{j = 1}^{m} (V(S_{lj0}) \al'_{lj0} \tilde D(\pi, \la_{lj0}, \la_{nk1}) \al_{nk1} - V(S_{lj1}) \al'_{lj1} \tilde D(\pi, \la_{lj1}, \la_{nk1}) \al_{nk1}).
\end{equation}

Using~\eqref{defD} and~\eqref{eqv} for the problem $\tilde L$, one can easily show that
\begin{equation} \label{intD}
\tilde D(x, \la, \mu) = \int_0^x \tilde S^{\dagger}(t, \bar \la) \tilde S(t, \mu) \, dt.
\end{equation}
It has been proved in \cite{Bond19-sp}, that
\begin{equation}
    \al_{lj1} \int_0^{\pi} \tilde S^{\dagger}(t, \la_{lj1}) \tilde S(t, \la_{nk1}) \, dt \, \al_{nk1} = \begin{cases}
                                             \al_{nk1}, \quad \la_{nk1} = \la_{lj1}, \\
                                             0, \quad \la_{nk1} \ne \la_{lj1}.
                                        \end{cases}    
\end{equation}
Combining the latter relations, we obtain that
$$
\sum_{l = 1}^{\iy} \sum_{j = 1}^m V(S_{lj1}) \al'_{lj1} \tilde D(\pi, \la_{lj1}, \la_{nk1}) \al_{nk1} = V(S_{nk1}) \al_{nk1}.
$$
Hence \eqref{VS3} takes the form
\begin{equation} \label{VS4}
\sum_{l = 1}^{\iy} \sum_{j = 1}^m V(S_{lj0}) \al'_{lj0} \tilde D(\pi, \la_{lj0}, \la_{nk1}) \al_{nk1} = 0.
\end{equation}
Substituting~\eqref{intD} into~\eqref{VS4}, we arrive at the relations
\begin{equation} \label{intfS}
    \int_0^{\pi} f(t) \tilde S_{nk1}(t) \al_{nk1} \, dt = 0, \quad n \ge 1, \quad k = \overline{1, m},
\end{equation}
where $f(t)$ is defined by~\eqref{deff}. By virtue of Lemma~\ref{lem:fL2}, $f \in L_2((0, \pi); \mathbb C^{m \times m})$. Then we apply Proposition~\ref{prop:basis}. The completeness of the sequence $\{ \tilde S_{nk1}(t) \tilde{\mathcal E}_{nk} \}_{n \ge 1, \, k = \overline{1, m}}$ together with~\eqref{intfS} implies $f = 0$ in $L_2((0, \pi); \mathbb C^{m \times m})$. The relation~\eqref{deff} and the minimality of the sequence $\{ \tilde S_{nk0}(x) \mathcal E_{nk} \}_{n \ge 1, \, k = \overline{1, m}}$ yield 
\begin{equation} \label{VSal}
    V(S_{lj0}) \al_{lj0} = 0, \quad l \ge 1, \quad j = \overline{1, m}.
\end{equation}

Differentiating~\eqref{defPhi}, we derive the following relation similar to~\eqref{VS1}:
$$
V(\Phi) = \tilde V(\tilde \Phi) - T \eps_0(\pi) T^{\perp} \tilde \Phi(\pi, \la) - \sum_{l = 1}^{\iy} \sum_{j = 1}^m (V(S_{lj0}) \al'_{lj0} \tilde E(\pi, \la_{lj0}, \la) - V(S_{lj1}) \al'_{lj1} \tilde E(\pi, \la_{lj1}, \la)).
$$
Recall that $\tilde V(\tilde \Phi) = 0$, $T^{\perp} \tilde \Phi(\pi, \la) = 0$. Using~\eqref{defE} and~\eqref{VS2}, one can easily show that $\al'_{lj1} \tilde E(\pi, \la_{lj1}, \la) = 0$ for $\la \ne \la_{lj1}$, $l \ge 1$, $j = \overline{1, m}$. Consequently, taking~\eqref{VSal} into account, we obtain that $V(\Phi) = 0$.
\end{proof}

Lemmas~\ref{lem:relSPhi} and~\ref{lem:VPhi} imply that $\Phi(x, \la)$ defined by~\eqref{defPhi} is the Weyl solution of the problem $L$. Hence $M(\la) := \Phi'(0, \la)$ is the Weyl matrix of $L$.
Using~\eqref{defPhi}, we derive the relation
$$
M(\la) = \tilde M(\la) - \sum_{l = 1}^{\iy} \sum_{j = 1}^m \left( \frac{\al'_{lj0}}{\la - \la_{lj0}} - \frac{\al'_{lj1}}{\la - \la_{lj1}}\right).
$$
Obviously, the singularities of $M(\la)$ coincide with $\{ \la_{nk} \}_{n \ge 1, \, k = \overline{1, m}}$ and the relation~\eqref{defal} holds, so $\{ \la_{nk}, \al_{nk} \}_{n \ge 1, \, k = \overline{1, m}}$ are the spectral data of $L$.

In order to finish the proof of Theorem~\ref{thm:sc}, it remains to show that the matrices $Q(x)$ and $H$ constructed by~\eqref{defQH} are Hermitian. For this purpose, along with $L$ we consider the boundary value problem $L^* = L^*(Q(x), T, H)$ of the following form:
\begin{gather} \label{eqvZ}
-Z''(x) + Z(x) Q(x) = \la Z(x), \quad x \in (0, \pi), \\ \label{bcZ}
Z(0) = 0, \quad V^*(Z) := (Z'(\pi) - Z(\pi) H) T - Z(\pi) T^{\perp} = 0.
\end{gather}

The Weyl solution of the problem $L^*$ is the matrix solution $\Phi^*(x, \la)$ of eq.~\eqref{eqvZ}, satisfying the conditions $\Phi^*(0, \la) = I$, $V^*(\Phi^*) = 0$. The Weyl matrix is defined as $M^*(\la) := {\Phi^*}'(0, \la)$. Using the approach of \cite{Yur06}, one can show that $M(\la) \equiv M^*(\la)$, so the eigenvalues $\{ \la_{nk}^* \}_{n \ge 1, \, k = \overline{1, m} }$ of the problem $L^*$ coincide with the eigenvalues of $L$, and the weight matrices $\al_{nk}^* := \Res\limits_{\la = \la_{nk}^*} M^*(\la)$ coincide with $\al_{nk}$, $n \ge 1$, $k = \overline{1, m}$.

Taking the conjugate transpose of~\eqref{eqvZ} and~\eqref{bcZ}, we conclude that the boundary value problem $L^{\dagger} := L(Q^{\dagger}(x), T, H^{\dagger})$ has the spectral data $\{ \bar \la_{nk}^*, (\al^*_{nk})^{\dagger} \}_{n \ge 1, \, k = \overline{1, m}} = \{ \bar \la_{nk}, \al_{nk}^{\dagger} \}_{n \ge 1, \, k = \overline{1, m}}$. Since $\la_{nk} \in \mathbb R$ and $\al_{nk} = \al_{nk}^{\dagger}$, $n \ge 1$, $k = \overline{1, m}$, the spectral data of the problems $L$ and $L^{\dagger}$ coincide. Uniqueness of Inverse Problem~\ref{ip:matr} solution (see \cite{Xu19, Bond19-alg}) yields that $L = L^{\dagger}$, i.e. $Q(x) = Q^{\dagger}(x)$ for a.a. $x \in (0, \pi)$ and $H = H^{\dagger}$. Theorem~\ref{thm:sc} is completely proved.

\begin{lem}
In the graph case, the matrix function $Q(x)$ constructed by~\eqref{defQH} is diagonal.
\end{lem}

\begin{proof}
Recall that, in the graph case, we have an additional requirement: the solution $\psi(x)$ of the main equation~\eqref{main} is diagonal. Consequently, the matrices $S_{ljs}(x)$ are diagonal for all $l \ge 1$, $j = \overline{1, m}$, $s = 0, 1$. The relation~\eqref{relSlj0}, the asymptotic formulas~\eqref{asymptla} and~\eqref{asymptS} together with interpolation arguments imply that the matrix function $S(x, \la)$ is also diagonal. In view of~\eqref{eqvS}, we conclude that $Q(x)$ is diagonal.
\end{proof}

Since in the graph case $\rank(T) = 1$, the matrix $H$ defined by~\eqref{defQH} has the form $H = h T$, $h \in \mathbb R$. According to~\eqref{defz1} and~\eqref{defP2}, the number $h$ is uniquely specified by the numbers $\{ z_j \}_{j = 1}^m$:
$$
h = \frac{1}{m} \sum_{j = 2}^m z_j - z_1.
$$
Since $z_j = \tilde z_j$, $j = \overline{1, m}$, we have $h = \tilde h$
($\tilde h$ is defined by~\eqref{modelg}). Thus, Theorem~\ref{thm:scg} is proved.

\section{Local Solvability and Stability}

This section concerns local solvability and stability of the studied inverse problems.  
The main results are formulated in Theorems~\ref{thm:loc} and~\ref{thm:locg} for the general matrix case and for the graph case, respectively. The proofs of these theorems strongly rely on the results of the previous sections.

Let $\tilde L = L(\tilde Q(x), T, \tilde H)$ be a fixed boundary value problem in the general form, and let $\tilde {\mathcal S} = \{ \tilde \la_{nk}, \tilde \al_{nk} \}_{n \ge 1, \, k = \overline{1, m}}$ be the spectral data of $\tilde L$. Denote by $SD(\tilde {\mathcal S})$ the set of all collections $\mathcal S = \{ \la_{nk}, \al_{nk} \}_{n \ge 1, \, k = \overline{1, m}} \in SD$, such that the relations~\eqref{asymptla}, \eqref{asymptal} and~\eqref{eqcoef} hold.

Let us group the numbers $\{ \tilde \rho_{nk} \}_{n \ge 1, \, k = \overline{1, m}}$ into the collections (multisets) as follows
$$
\tilde G_0 := \{ \tilde \rho_{nk} \}_{n = \overline{1, n_0}, \, k = \overline{1, m}}, \quad \tilde G_{2j} := \{ \tilde \rho_{n_0 + j, k} \}_{k = 1}^p, \quad \tilde G_{2j + 1} := \{ \tilde \rho_{n_0 + j, k} \}_{k = p + 1}^m, \quad j \ge 1,
$$
where $n_0$ is a fixed integer such that $\tilde G_n \cap \tilde G_k = \varnothing$ for all $n \ne k$. 

Let $\mathcal S$ be an arbitrary element from $SD(\tilde {\mathcal S})$. Consider the partition $\{ G_k \}_{k \ge 1}$ of the numbers $\{ \rho_{nks} \}_{n \ge 1, \, k = \overline{1, m}, \, s = 0, 1}$, defined by~\eqref{defG} and having the same $n_0$ as the partition $\{ \tilde G_k \}_{k \ge 1}$. Clearly, $\tilde G_k \subset G_k$, $k \ge 1$. 
Consider an arbitrary partition $G_k = \bigcup\limits_{i = 1}^{p_k} G_{ki}$, $p_k \in \mathbb N$, $k \ge 1$ (which not necessarily coincide with the one defined in Section~4), satisfying the conditions: $G_{ki} \cap G_{kj} = \varnothing$ for all $i = \overline{1, p_k}$, $k \ge 1$, and $\rho_{lj0} \in G_{ki}$ iff $\rho_{lj1} \in G_{ki}$. 
For any such partition, we define the numbers $\{ \xi_k \}_{k \ge 1}$ and $\Xi$ by the formulas~\eqref{defxi} and~\eqref{defXi}, respectively. For a collection $\mathcal S$, we fix such partition $\{ G_{ki} \}_{k \ge 1, \, i = \overline{1, p_k}}$ that the value of $\Xi$ is minimal possible.

The following theorem gives local solvability and stability of Inverse Problem~\ref{ip:matr}.

\begin{thm} \label{thm:loc}
Let $\tilde L = L(\tilde Q(x), T, \tilde H)$ is a fixed boundary value problem in the general form, and let $\tilde {\mathcal S}$ be the spectral data of $\tilde L$. Then there exists $\de > 0$ (depending only on $\tilde L$) such that for any data $\mathcal S \in SD(\tilde {\mathcal S})$ such that $\Xi \le \de$ there exist a unique matrix function $Q(x) = Q^{\dagger}(x) \in L_2((0, \pi); \mathbb C^{m \times m})$ and a unique matrix $H = H^{\dagger}$ such that the problem $L = L(Q(x), T, H)$ has the spectral data $\mathcal S$ and
\begin{equation} \label{estQH}
    \| Q - \tilde Q \|_{L_2} \le C \Xi, \quad \| H - \tilde H \| \le C \Xi. 
\end{equation}
The constant $C$ in the estimates~\eqref{estQH} depends only on $\tilde L$, $\de$ and not on a particular choice of $\mathcal S$.
\end{thm}

\begin{proof}
Fix $\de_0$ so that, for every $\mathcal S \in SD(\tilde {\mathcal S})$ satisfying $\Xi \le \de_0$, we have $G_n \cap G_k = \varnothing$, $n \ne k$, $n, k \ge 1$. By using the partition $\{ G_k \}_{k \ge 1}$, we construct the Banach space $B$ and the operator $\tilde R(x)$, as it was described in Section~2. Then the estimates of Proposition~\ref{prop:estR} hold with a constant $C$ depending only on $\tilde L$, $n_0$ and $\de_0$ and independent of $\mathcal S$ and $x$. Consequently, one can choose $\de \in (0, \de_0]$ such that $\Xi \le \de$ implies $\| \tilde R(x) \|_{B \to B} \le \frac{1}{2}$. Then the operator $(\mathcal I + \tilde R(x))$ has a bounded inverse in $B$. Hence the main equation~\eqref{main} constructed by $\tilde L$ and $\mathcal S$ is uniquely solvable. Relying on the results of Section~6, we conclude that $\mathcal S$ is the spectral data of a unique boundary value problem $L = L(Q(x), T, H)$.
The estimates~\eqref{estQH} easily follow from~\eqref{defQH} and the estimate $\| \eps \|_{L_2} \le C \Xi$ of Lemma~\ref{lem:eps}.
\end{proof}

Let us compare Theorems~\ref{thm:loc} and~\ref{thm:sc}. Theorem~\ref{thm:loc} has a local nature, while Theorem~\ref{thm:sc} establishes global solvability of Inverse Problem~\ref{ip:matr}. However, the advantage of Theorem~\ref{thm:loc} is that \textsc{Completeness} condition is not required. 

Theorem~\ref{thm:loc} yields the following corollary.
Let $\{ \la_{nk}^0, \al_{nk}^0 \}_{n \ge 1, \, k = \overline{1, m}}$ be the spectral data of the problem $L(\tfrac{2}{\pi} \tilde \Omega, T, \tilde H)$. For every $N \in \mathbb N$, define the data $\mathcal S^N = \{ \la_{nk}^N, \al_{nk}^N \}_{n \ge 1, \, k = \overline{1, m}}$, $N \in \mathbb N$, as follows
$$
\la_{nk}^N = \begin{cases}
                \tilde \la_{nk}, \quad n \le N, \\
                \la_{nk}^0, \quad n > N,
             \end{cases}
\al_{nk}^N = \begin{cases}
                \tilde \al_{nk}, \quad n \le N, \\
                \al_{nk}^0, \quad n > N.
             \end{cases}
$$

Obviously, $\mathcal S^N \in SD(\tilde {\mathcal S})$ for sufficiently large $N$.

\begin{cor} \label{cor:locN}
Let $\tilde L = L(\tilde Q(x), T, \tilde H)$ be a fixed boundary value problem in the general form, and let $\tilde {\mathcal S}$ be the spectral data of $\tilde L$. Then for sufficiently large $N$ the data $\mathcal S^N$ is the spectral data of a certain boundary value problem $L^N = L(Q^N(x), T, H^N)$ in the general form, and
$$
\lim_{N \to \iy} \| Q^N - \tilde Q \|_{L_2} = 0, \quad \lim_{N \to \iy} \| H^N - \tilde H \| = 0.
$$
\end{cor}

Corollary~\ref{cor:locN} is important for numerical solution of Inverse Problem~\ref{ip:matr}, since only a finite number of spectral data $\{ \la_{nk}, \al_{nk} \}_{n = \overline{1, N}, \, k = \overline{1, m}}$ is usually available in practice. 

Theorem~\ref{thm:loc} also yields the following corollary.

\begin{thm} \label{thm:locg}
Let $\tilde L = L(\tilde Q(x), T, \tilde H)$ is a fixed boundary value problem in the graph case, and let $\tilde {\mathcal S}$ be the spectral data of $\tilde L$. Then there exists $\de > 0$ depending only on $\tilde L$ such that for any data $\mathcal S \in SD(\tilde {\mathcal S})$, such that $\Xi \le \de$ and the solution of the main equation~\eqref{main} is diagonal, there exists a unique real-valued matrix function $Q(x) = \diag\{ q_j(x) \}_{j = 1}^m \in L_2((0, \pi); \mathbb C^{m \times m})$ such that the problem $L = L(Q(x), T, H)$, $H = \tilde H$, has the spectral data $\mathcal S$, and
\begin{equation} \label{estqj}
\| q_j - \tilde q_j \|_{L_2} \le C \Xi, \quad j = \overline{1, m}.
\end{equation}
The constant $C$ in the estimate~\eqref{estqj} depends only on $\tilde L$, $\de$ and not  on a particular choice of $\mathcal S$.
\end{thm}

Theorem~\ref{thm:locg} asserts local solvability and stability of Inverse Problem~\ref{ip:graph}.
It follows from the proof of Theorem~\ref{thm:loc}, that for sufficiently small $\Xi$ the main equation~\eqref{main} is automatically solvable. Therefore we do not need to require \textsc{Solvability} condition in Theorem~\ref{thm:locg}, so only \textsc{Diagonality} is additionally required. The equality $H = \tilde H$ in the graph case follows from the spectral data asymptotics.

\medskip

{\bf Acknowledgement.} This work was supported by Grant 19-71-00009 of the Russian Science Foundation.

\noindent Natalia Pavlovna Bondarenko \\
1. Department of Applied Mathematics and Physics, Samara National Research University, \\
Moskovskoye Shosse 34, Samara 443086, Russia, \\
2. Department of Mechanics and Mathematics, Saratov State University, \\
Astrakhanskaya 83, Saratov 410012, Russia, \\
e-mail: {\it BondarenkoNP@info.sgu.ru}


\medskip

\begin{thebibliography}{99}


\bibitem{Nic85}
Nicaise, S. Some results on spectral theory over networks, applied to nerve impulse transmission, Vol. 1771, Lecture notes in mathematics. Berlin: Springer (1985), 532-–541.

\bibitem{LLS94}
Langese, J.; Leugering, G.; Schmidt, J. Modelling, analysis and control of dynamic elastic multi-link structures. Birkh\"{a}user, Boston (1994).

\bibitem{Kuch02}
Kuchment, P. Graph models for waves in thin structures, Waves in Random Media 12 (2002), no.~4, R1--R24.

\bibitem{GS06}
Gnutzmann, S.; Smilansky, U. Quantum graphs: applications to quantum chaos and universal spectral statistics, Adv. Phys. 55 (2006), 527--625.

\bibitem{BCFK06}
Berkolaiko, G.; Carlson, R.; Fulling, S.; Kuchment, P. Quantum Graphs and Their Applications, Contemp. Math. 415, Amer. Math. Soc., Providence, RI (2006).

\bibitem{BK13}
Berkolaiko, G.; Kuchment, P. Introduction to Quantum Graphs, Amer. Math. Soc., Providence, RI (2013).

\bibitem{PPP04}
Pokorny, Yu. V.; Penkin, O. M.; Pryadiev, V. L. et al. Differential Equations on Geometrical Graphs, Fizmatlit, Moscow (2004) (Russian).


\bibitem{Piv00}
Pivovarchik, V. N. Inverse problem for the Sturm-Liouville equation on a simple graph, SIAM J. Math. Anal. 32 (2000), no.~4, 801--819.

\bibitem{KS00} 
Kostrykin, V.; Schrader, R. Kirchhoff's rule for quantum wires. II: The inverse problem with possible applications to quantum computers,
Fortschritte der Physik 48 (2000), 703--716. 

\bibitem{GS01} 
Gutkin, B.; Smilansky, U. Can one hear the shape of a graph? J. Phys. A, 34 (2001), no.~31, 6061--6068.

\bibitem{Har02}
Harmer, M. Inverse scattering for the matrix Schr\"{o}dinger operator and Schr\"{o}dinger operator on graphs with general self-adjoint boundary conditions, ANZIAM J. 43 (2002), 1--8.

\bibitem{Bel04}
Belishev, M. I. Boundary spectral inverse problem on a class of graphs (trees) by the BC-method, Inverse Problems 20 (2004), 647-–672.

\bibitem{KN05} 
Kurasov, P.; Nowaczyk, M. Inverse spectral problem for quantum graphs, J. Phys. A 38 (2005), no.~22, 4901-–4915.

\bibitem{BW05}
Brown, B. M.; Weikard, R. A Borg-Levinson theorem for trees, Proceedings of the Royal Society A: Mathematical, Physical and Engineering Sciences 461 (2005), 3231--3243.

\bibitem{Yur05}
Yurko, V.A. Inverse spectral problems for Sturm-Liouville operators on graphs, Inverse Problems 21 (2005), 1075--1086.

\bibitem{BV06}
Belishev, M.I.; Vakulenko, A.F. Inverse problems on graphs: recovering the tree of strings by the BC-method, J. Inverse Ill-Posed Probl. 14 (2006), 29--46.

\bibitem{TMM06}
Trooshin, I.; Marchenko, V.; Mochizuki, K. Inverse scattering on a graph containing circle, Analytic methods of analysis and DEs: AMADE 2006, 237--243, Camb. Sci. Publ., Cambridge (2008).
    
\bibitem{Piv07}
Pivovarchik, V. N. Inverse problem for the Sturm-Liouville equation on a star-shaped graph, Math. Nachr. 280 (2007), 1595–-1619.

\bibitem{CP07}
Carlson, R.; Pivovarchik, V. Ambarzumian's theorem for trees, Electronic J. Diff. Eqns. 2007 (2007), no.~142, 1--9.

\bibitem{AK08}
Avdonin, S.; Kurasov, P. Inverse problems for quantum trees, Inv. Probl. Imag. 2 (2008), no.~1, 1--21. 

\bibitem{FIY08}
Freiling, G.; Ignatiev, M. Y.; Yurko, V. A. An inverse spectral problem for Sturm-Liouville operators with singular potentials on star-type graph, Proc. Symp. Pure Math. 77 (2008), 397--408.

\bibitem{Kur10}
Kurasov, P. Inverse problems for Aharonov-Bohm rings, Math. Proc. Cambridge Philos. Soc. 148 (2010), 331--362.

\bibitem{Yang10}
Yang, C.-F. Inverse spectral problems for the Sturm-Liouville operator on a $d$-star graph, J. Math. Anal. Appl. 365 (2010), 742--749.

\bibitem{ALM10} 
Avdonin, S.; Leugering, G.; Mikhaylov, V. On an inverse problem for tree-like networks of elastic strings, Z. Angew. Math. Mech. 90 (2010), no.~2, 136–150.

\bibitem{AKN10} 
Avdonin, S.; Kurasov, P.; Nowaczyk, M. Inverse Problems for Quantum
Trees II. Recovering Matching Conditions for Star Graphs, Inv. Probl. Imag., 4 (2010), no.~4, 579--598.

\bibitem{EK12} 
Ershova, Yu.; Kiselev, A. V. Trace formulae for graph Laplacians with applications to recovering matching conditions, Methods Funct. Anal. Topology 18 (2012), no.~4, 343--359.

\bibitem{BF13}
Buterin, S.A.; Freiling, G. Inverse spectral-scattering problem for the Sturm-Liouville operator on a noncompact star-type graph, Tamkang J. Math. 44 (2013), no.~3, 327--349.

\bibitem{Ign15}
Ignatiev, M. Inverse scattering problem for Sturm-Liouville operator on non-compact A-graph. Uniqueness result, Tamkang J. Math. 46 (2015), no.~4, 401--422.

\bibitem{Yur16}
Yurko, V. A. Inverse spectral problems for differential operators on spatial networks, Russian Mathematical Surveys 71 (2016), no.~3, 539--584.

\bibitem{BS17}
Bondarenko, N.; Shieh, C.-T. Partial inverse problems for Sturm-Liouville operators on trees, Proc. Royal Soc. Edinburgh Section A: Mathematics 147A (2017), 917-933

\bibitem{MT17}
Mochizuki, K.; Trooshin, I. On inverse scattering on a sun-type graph, New Trends in Analysis and Interdisciplinary Applications (2017), 319--325.

\bibitem{XY18}
Xu, X.-C.; Yang, C.-F. Inverse scattering problems on a noncompact star graph, Inverse problems 34 (2018), no.~11, 12pp.


\bibitem{Mar77}
Marchenko, V.A. Sturm-Liouville Operators and Their Applications, Naukova Dumka, Kiev (1977) (Russian); English transl., Birkhauser (1986).

\bibitem{Lev84}
Levitan, B.M. Inverse Sturm-Liouville Problems, Nauka, Moscow (1984) (Russian); English transl., VNU Sci. Press, Utrecht (1987).

\bibitem{PT87}
P\"{o}schel, J.; Trubowitz, E. Inverse Spectral Theory, New York, Academic Press (1987).

\bibitem{FY01}
Freiling, G.; Yurko, V. Inverse Sturm-Liouville Problems and Their Applications, Huntington, NY: Nova Science Publishers (2001).

\bibitem{AM63}
Agranovich, Z. S.; Marchenko, V. A. The inverse problem of scattering theory, Gordon and Breach, New York (1963).

\bibitem{Xu19}
Xu, X.-C. Inverse spectral problem for the matrix Sturm-Liouville operator with the general separated self-adjoint boundary conditions, Tamkang J. Math. 50 (2019), no.~3, 321--336.

\bibitem{Bond19-sp}
Bondarenko, N. P. Spectral theory of the Sturm-Liouville operator on the star-shaped graph, Math. Meth. Appl. Sci. 43 (2020), no. 2, 471--485.

\bibitem{Bond19-alg}
Bondarenko, N. P. Constructive solution of the inverse spectral problem for the matrix Sturm-Liouville operator, Inv. Probl. Sci. Eng. 28 (2020), no. 9, 1307--1330.

\bibitem{CK09}
Chelkak, D.; Korotyaev, E. Weyl-Titchmarsh functions of vector-valued Sturm-Liouville operators on the unit interval, J. Func. Anal. 257 (2009), 1546--1588.

\bibitem{MT10} 
Mykytyuk, Ya.V.; Trush, N.S. Inverse spectral problems for Sturm-Liouville
operators with matrix-valued potentials, Inverse Problems 26 (2010), 015009.

\bibitem{Bond12}
Bondarenko, N.P. Necessary and sufficient conditions for the solvability of the inverse problem for the matrix Sturm-Liouville operator, 
Funct. Anal. Appl. 46 (2012), no.~1, 53--57.

\bibitem{Bond19-tamkang}
Bondarenko, N. P. An inverse problem for the non-self-adjoint matrix Sturm-Liouville operator, Tamkang J. Math. 50 (2019), no.~1, 71--102. 

\bibitem{Bond19-asympt} 
Bondarenko N.P. Spectral analysis of the matrix Sturm-Liouville operator, Boundary Value Problems (2019), 2019:178.

\bibitem{MP15}
M\"oller, M.; Pivovarchick, V. Spectral Theory of Operator Pencils, Hermite-Biehler Functions, and their Applications,
Operator Theory: Advances and Applications, Vol. 246. Birkh\"auser, Basel (2015).

\bibitem{Yur06}
Yurko, V. Inverse problems for the matrix Sturm-Liouville equation on a finite interval, Inverse Problems 22 (2006), 1139--1149.

\end{thebibliography}
\end{document}